\documentclass[twoside, 11pt, leqno]{amsart}

\usepackage{amsfonts}
\usepackage[utf8]{inputenc}
\usepackage{array}

\usepackage[english]{babel}
\usepackage{latexsym}
\usepackage{amscd}
\usepackage{amsmath}
\usepackage{amssymb}
\usepackage{amsthm}
\usepackage{graphicx}
\usepackage{mathrsfs}
\usepackage{mathtools}
\usepackage{pict2e}
\usepackage{stackrel} 
\usepackage{wasysym}
\usepackage[all]{xy}
\usepackage[dvipsnames]{xcolor}
\usepackage[colorlinks,final,hyperindex]{hyperref}
\usepackage[noabbrev,capitalize]{cleveref}
\usepackage{tikz}
\usepackage{tikz-cd}
\usepackage{enumitem}
\usetikzlibrary{decorations.pathmorphing,decorations.markings,arrows,calc,shapes.geometric,arrows.meta,positioning}
\usepackage{geometry}

\allowdisplaybreaks

\newtheorem{lemma}{Lemma}[section]

\newtheorem{theorem}[lemma]{Theorem}

\setlist[enumerate]{label=\textnormal{(\arabic*)}, ref=\arabic*}

\newtheorem{corollary}[lemma]{Corollary}
\newtheorem{proposition}[lemma]{Proposition}

\theoremstyle{definition}
\newtheorem{definition}[lemma]{Definition}
\newtheorem{remark}[lemma]{\sc Remark}
\newtheorem{example}[lemma]{\sc Example}

\newtheorem{construction}[lemma]{\sc Construction}

\newtheorem*{notation}{Notation}
\numberwithin{equation}{section}

\newcommand{\nocontentsline}[3]{}
\newcommand{\tocless}[2]{\bgroup\let\addcontentsline=\nocontentsline#1{#2}\egroup}

\hypersetup{citecolor=BleuMinuit, linkcolor=bordeau, filecolor=black, urlcolor=BleuMinuit}

\author{Coline Emprin}
\title{Obstruction sequences to homotopy equivalences}
\address{Coline Emprin, Department of Mathematics, Stockholm University, Albanovägen 28, 106 91, \textcolor{white}{tt} Stockholm, Sweden}

\email{\noindent \href{mailto:coline.emprin@ens.psl.eu}{coline.emprin@math.su.se}}

\date{\today}



\begin{document}
\definecolor{red}{RGB}{230,97,0}
\definecolor{blue}{RGB}{93,58,155}
\definecolor{Chocolat}{rgb}{0.36, 0.2, 0.09}
\definecolor{BleuTresFonce}{rgb}{0.215, 0.215, 0.36}
\definecolor{BleuMinuit}{RGB}{0, 51, 102}
\definecolor{bordeau}{rgb}{0.5,0,0}
\definecolor{BORDEAU}{rgb}{0.5,0,0}
\definecolor{turquoise}{RGB}{6, 62, 62}
\definecolor{rose}{RGB}{235, 62, 124}

\newcommand{\coline}[1]{\textcolor{rose}{#1}}

\newcommand{\hooklongrightarrow}{\lhook\joinrel\longrightarrow}


\newcommand{\chdiscr}{\mathsf{discr.\ Ch}}
\newcommand{\sSe}{\mathsf{sSet}}
\newcommand{\sLialg}{\ensuremath{\mathcal{sL}_\infty\text{-}\,\mathsf{alg}}}
\newcommand{\isLialg}{\ensuremath{\infty\text{-}\,\mathcal{sL}_\infty\text{-}\,\mathsf{alg}}}
\newcommand{\sLiealg}{\ensuremath{\mathcal{s}\lie\,\text{-}\,\mathsf{alg}}}
\def\cD{\Delta}

\newcommand{\ucomnalg}{\ensuremath{\mathrm{uCom}_{\leslant 0}\text{-}\,\mathsf{alg}}}
\newcommand{\fnQsp}{\ensuremath{\mathsf{ft}\mathbb{Q}\text{-}\mathsf{ho}(\mathsf{sSet})}}
\newcommand{\fnsp}{\ensuremath{\mathsf{fn}\text{-}\mathsf{Sp}}}
\newcommand{\nQsp}{\ensuremath{\mathsf{n}\mathbb{Q}\text{-}\mathsf{Sp}}}
\newcommand{\nsp}{\ensuremath{\mathsf{n}\text{-}\mathsf{Sp}}}
\newcommand{\fQalg}{\ensuremath{\mathsf{ft}\text{-}\mathsf{ho}(\mathrm{uCom}_{\leqslant 0}\text{-}\,\mathsf{alg})}}

\newcommand{\coker}{\operatorname{coker}}
\newcommand{\Ran}{\operatorname{Ran}}
\newcommand{\ho}{\operatorname{ho}}
\newcommand{\Mal}{\operatorname{Mal}}
\newcommand{\Loc}{\operatorname{Loc}}


\newcommand{\sLi}{\ensuremath{\mathcal{sL}_\infty}}
\newcommand{\ccoLi}{\ensuremath{\mathrm{cBCom}}}
\newcommand{\sLie}{\ensuremath{\mathcal{s}\mathrm{Lie}}}
\newcommand{\com}{\ensuremath{\mathrm{Com}}}
\newcommand{\ucom}{\ensuremath{\mathrm{uCom}}}
\newcommand{\Cobar}{\ensuremath{\Omega}}
\newcommand{\hatCobar}{\ensuremath{\widehat{\Omega}}}
\renewcommand{\Bar}{\ensuremath{\mathrm{B}}}
\newcommand{\hatBar}{\ensuremath{\widehat{\mathrm{B}}}}
\newcommand{\lie}{\ensuremath{\mathrm{Lie}}}
\newcommand{\uhocom}{\ensuremath{\mathrm{u}\Omega\mathrm{BCom}}}

\newcommand{\Sh}{\ensuremath{\mathrm{Sh}}}


\newcommand{\g}{\ensuremath{\mathfrak{g}}}
\newcommand{\wsLi}{\ensuremath{\widehat{\mathcal{sL}_\infty}}}


\newcommand{\Rh}{\ensuremath{\mathrm{R^h}}}
\renewcommand{\L}{\ensuremath{\mathscr{L}}}
\newcommand{\Li}{\mathfrak{L}}
\newcommand{\APL}{\mathrm{A_{PL}}}
\newcommand{\CPL}{\mathrm{C_{PL}}}


\newcommand{\antishriek}{\text{\raisebox{\depth}{\textexclamdown}}}
\newcommand{\RT}{\mathrm{RT}}
\newcommand{\LRT}{\mathrm{LRT}}
\newcommand{\Sy}{\mathbb{S}}
\renewcommand{\d}{\ensuremath{\mathrm{d}}}
\newcommand{\PP}{\ensuremath{\mathrm{P}}}
\def\Ho#1#2{\Lambda^{#2}_{#1}}
\def\De#1{\Delta^{#1}}
\newcommand{\NN}{\mathbb{N}}
\newcommand{\RR}{\mathbb{R}}
\def\BCH{\mathrm{BCH}}
\newcommand{\wPT}{\ensuremath{\mathrm{wPT}}}
\newcommand{\oPT}{\ensuremath{\overline{\mathrm{PT}}}}
\newcommand{\PaRT}{\ensuremath{\mathrm{PaRT}}}
\newcommand{\PaPT}{\ensuremath{\mathrm{PaPT}}}
\newcommand{\PaPRT}{\ensuremath{\mathrm{PaPRT}}}
\def\rmC{\mathrm{C}}
\newcommand{\berglund}{\ensuremath{\mathcal{B}}}
\def\hot{\widehat{\otimes}} 
\def\whk{\widehat{k}}

\def\colim{\mathop{\mathrm{colim}}}

\newcommand{\CC}{\ensuremath{\mathrm{CC}_\infty}}


\newcommand{\Lalg}{\ensuremath{\mathscr{L}_\infty\text{-}\mathsf{alg}}}

\newcommand{\F}{\ensuremath{\mathrm{F}}}
\newcommand{\h}{\ensuremath{\mathfrak{h}}}

\newcommand{\QQ}{\ensuremath{\mathbb{Q}}}
\newcommand{\D}{\ensuremath{\mathscr{D}}}
\renewcommand{\P}{\ensuremath{\mathcal{P}}}
\newcommand{\C}{\ensuremath{\mathcal{C}}}
\newcommand{\VdL}{\ensuremath{\mathrm{VdL}}}
\newcommand{\ch}{\ensuremath{\mathrm{Ch}}}
\newcommand{\End}{\ensuremath{\mathrm{End}}}
\newcommand{\Aut}{\ensuremath{\mathrm{Aut}}}
\newcommand{\eend}{\ensuremath{\mathrm{end}}}
\newcommand{\susp}{\ensuremath{\mathscr{S}}}
\newcommand{\T}{\ensuremath{\mathcal{T}}}
\newcommand{\vdl}{\ensuremath{\mathrm{VdL}}}
\newcommand{\Tw}{\ensuremath{\mathrm{Tw}}}
\newcommand{\Hom}{\ensuremath{\mathrm{Hom}}}
\renewcommand{\S}{\ensuremath{\mathbb{S}}}
\renewcommand{\k}{\ensuremath{\mathbb{K}}}
\newcommand{\id}{\ensuremath{\mathrm{id}}}
\newcommand{\MC}{\ensuremath{\mathrm{MC}}}
\newcommand{\mc}{\ensuremath{\mathfrak{mc}}}
\newcommand{\mclie}{\ensuremath{\overline{\mathfrak{mc}}}}
\newcommand{\dgl}{\ensuremath{\mathsf{dgLie}}}

\newcommand{\B}{\mathcal{B}}
\newcommand{\Z}{\mathbb{Z}}

\newcommand{\ad}{\operatorname{ad}}
\newcommand{\PTt}{\ensuremath{\widetilde{\mathrm{PT}}}}
\newcommand{\PT}{\ensuremath{\mathrm{PT}}}

\newcommand{\A}{\ensuremath{\mathrm{A}}}

\makeatletter

\newcommand{\R}{\mathbb{R}}
\newcommand{\K}{\mathcal{K}}
\newcommand{\N}{\mathbb{N}}
\newcommand{\calA}{\mathcal{A}}
\newcommand{\calB}{\mathcal{B}}
\newcommand{\calC}{\mathcal{C}}
\newcommand{\calD}{\mathcal{D}}
\newcommand{\calE}{\mathcal{E}}
\newcommand{\calF}{\mathcal{F}}
\newcommand{\calG}{\mathcal{G}}
\newcommand{\calH}{\mathcal{H}}
\newcommand{\calI}{\mathcal{I}}
\newcommand{\calJ}{\mathcal{J}}
\newcommand{\calK}{\mathcal{K}}
\newcommand{\calL}{\mathcal{L}}
\newcommand{\calM}{\mathcal{M}}
\newcommand{\calN}{\mathcal{N}}
\newcommand{\calO}{\mathcal{O}}
\newcommand{\calP}{\mathcal{P}}
\newcommand{\calQ}{\mathcal{Q}}
\newcommand{\calR}{\mathcal{R}}
\newcommand{\calS}{\mathcal{S}}
\newcommand{\calT}{\mathcal{T}}
\newcommand{\calU}{\mathcal{U}}
\newcommand{\calV}{\mathcal{V}}
\newcommand{\calW}{\mathcal{W}}
\newcommand{\calX}{\mathcal{X}}
\newcommand{\calY}{\mathcal{Y}}
\newcommand{\calZ}{\mathcal{Z}}
\newcommand{\kk}{{\Bbbk}}

\newcommand{\ev}{\mathop{\mathrm{ev}}\nolimits}
\newcommand{\co}{\mathop{\mathrm{co}}\nolimits}
\newcommand{\pt}{\mathop{\mathrm{pt}}\nolimits}
\newcommand{\wt}{\mathop{\mathrm{wt}}\nolimits}
\newcommand{\pCY}{\mathop{\mathrm{pCY}}\nolimits}
\newcommand{\gr}{\mathop{\mathrm{gr}}\nolimits}

\newcommand{\Imm}{\mathop{\mathrm{Im}}\nolimits}

\tikzset{>={Stealth[scale=1.2]}}
\tikzset{->-/.style={decoration={
			markings,
			mark=at position #1 with {\arrow{>}}},postaction={decorate}}}
\tikzset{-w-/.style={decoration={
			markings,
			mark=at position #1 with {\arrow{Stealth[fill=white,scale=1.4]}}},postaction={decorate}}}
\tikzset{->-/.default=0.65}
\tikzset{-w-/.default=0.65}
\tikzstyle{bullet}=[circle,fill=black,inner sep=0.5mm]
\tikzstyle{circ}=[circle,draw=black,fill=white,inner sep=0.5mm]
\tikzstyle{vertex}=[circle,draw=black,thick,inner sep=0.5mm]
\tikzstyle{dot}=[draw,circle,fill=black,minimum size=0.5mm,inner sep = 0mm, outer sep = 0mm]
\tikzset{darrow/.style={double distance = 4pt,>={Implies},->},
	darrowthin/.style={double equal sign distance,>={Implies},->},
	tarrow/.style={-,preaction={draw,darrow}},
	qarrow/.style={preaction={draw,darrow,shorten >=0pt},shorten >=1pt,-,double,double
		distance=0.2pt}}
\newcommand\tikcirc[1][2.5]{\tikz[baseline=-#1]{\draw[thick](0,0)circle[radius=#1mm];}}
\newcommand{\tikzfig}[1]{\begin{tikzpicture}[auto,baseline={([yshift=-.5ex]current bounding box.center)}]#1\end{tikzpicture}}
\usetikzlibrary{decorations.pathreplacing}

\newcommand{\DD}{\EuScript D}
\newcommand{\BB}{\EuScript B}
\newcommand{\OO}{\EuScript O}
\newcommand{\KK}{\EuScript K}
\newcommand{\YY}{\EuScript Y}
\newcommand\nn         {\nonumber \\}
\newcommand{\pf}{{\it Proof.}\hspace{2ex}}
\newcommand{\apf}{{\it Another Proof.}\hspace{2ex}}
\newcommand{\epf}{\hspace*{\fill}\emph{$\square$}}
\newcommand{\halmos}{\rule{1ex}{1.4ex}}
\newcommand{\pfbox}{\hspace*{\fill}\emph{$\halmos$}}

\newcommand{\ac}{\mathop{ac}\nolimits}
\newcommand{\sC}{\mathop{s^{-1}\overline{C}}\nolimits}
\newcommand{\sA}{\mathop{s\overline{A}}\nolimits}
\newcommand{\Rep}{\mathop{\mathrm{Rep}}\nolimits}
\newcommand{\Span}{\mathop{\mathrm{Span}}\nolimits}
\newcommand{\proj}{\mathop{\mathrm{proj}}\nolimits}
\newcommand{\PHom}{\mathop{\mathrm{PHom}}\nolimits}
\newcommand{\Dim}{\mathop{\mathrm{Dim}}\nolimits}

\newcommand{\Aus}{\mathop{\mathrm{Aus}}\nolimits}
\newcommand{\Ann}{\mathop{\mathrm{Ann}}\nolimits}
\newcommand{\APC}{\mathop{\mathrm{APC}}\nolimits}
\newcommand{\Barr}{\mathop{\mathrm{Bar}}\nolimits}
\newcommand{\Cone}{\mathop{\mathrm{Cone}}\nolimits}
\newcommand{\modu}{\mathop{\mathrm{mod}}\nolimits}
\newcommand{\DGMod}{\mathop{\mathrm{DGMod}}\nolimits}
\newcommand{\soc}{\mathop{\mathrm{soc}}\nolimits}
\newcommand{\dg}{\mathop{\mathrm{dg}}\nolimits}
\newcommand{\red}{\mathop{\mathrm{red}}\nolimits}
\newcommand{\Map}{\mathop{\mathrm{Map}}\nolimits}
\newcommand{\inj}{\mathop{\mathrm{inj}}\nolimits}
\newcommand{\op}{\mathop{\mathrm{op}}\nolimits}
\newcommand\M          {{\mathcal{M}}}
\newcommand{\Ker}{\mathop{\mathrm{Ker}}\nolimits}
\newcommand{\Tor}{\mathop{\mathrm{Tor}}\nolimits}
\newcommand{\Tot}{\mathop{\mathrm{Tot}}\nolimits}
\newcommand{\Ext}{\mathop{\mathrm{Ext}}\nolimits}
\newcommand{\sg}{\mathop{\mathrm{sg}}\nolimits}
\newcommand{\Sl}{\mathop{\mathfrak{sl}}\nolimits}
\newcommand{\nc}{\mathop{\mathrm{nc}}\nolimits}
\newcommand{\Tr}{\mathop{\mathrm{Tr}}\nolimits}
\newcommand{\Irr}{\mathop{\mathrm{Irr}}\nolimits}
\newcommand{\HC}{\mathop{\mathrm{HC}}\nolimits}
\newcommand{\HH}{\mathop{\mathrm{HH}}\nolimits}
\newcommand{\THH}{\mathop{\mathrm{TH}}\nolimits}
\newcommand{\Perf}{\mathop{\mathrm{Perf}}\nolimits}
\newcommand{\del}{\partial}

\newcommand{\verteq}{\rotatebox{90}{$\;\;=\;\;$}}

	\begin{abstract}
We develop an obstruction theory for the existence of gauge equivalences in complete differential graded Lie algebras. Specifically, this theory provides a characterization of homotopy equivalences between differential graded algebras governed by operads or properads, potentially colored in a groupoid.  We apply this framework to establish new homotopy equivalence results in both algebraic topology and algebraic geometry, with a particular focus on the study of minimal models for highly connected varieties.
\end{abstract}


	\maketitle
	

\makeatletter
\def\@tocline#1#2#3#4#5#6#7{\relax
	\ifnum #1>\c@tocdepth 
	\else
	\par \addpenalty\@secpenalty\addvspace{#2}%
	\begingroup \hyphenpenalty\@M
	\@ifempty{#4}{%
		\@tempdima\csname r@tocindent\number#1\endcsname\relax
	}{%
		\@tempdima#4\relax
	}%
	\parindent\z@ \leftskip#3\relax \advance\leftskip\@tempdima\relax
	\rightskip\@pnumwidth plus4em \parfillskip-\@pnumwidth
	#5\leavevmode\hskip-\@tempdima
	\ifcase #1
	\or\or \hskip 1em \or \hskip 2em \else \hskip 3em \fi%
	#6\nobreak\relax
	\hfill\hbox to\@pnumwidth{\@tocpagenum{#7}}\par
	\nobreak
	\endgroup
	\fi}

\newcommand{\enableopenany}{%
	\@openrightfalse%
}
\makeatother
	
	\setcounter{tocdepth}{1}
	\tableofcontents
	
	
	

\section*{\textcolor{bordeau}{Introduction}}

\noindent  Rational homotopy theory studies the rational homotopy type of topological spaces. Topological spaces $X$ and $Y$ have the same rational homotopy type if their rational cochain algebras are \emph{homotopy equivalent} as differential graded associative algebras, i.e. connected by a zig-zag    \[ C^*(X; \mathbb{Q}) \overset{\sim}{\longleftarrow} \cdot \overset{\sim}{\longrightarrow} \cdots \overset{\sim}{\longleftarrow} \cdot \overset{\sim}{\longrightarrow} C^*(Y; \mathbb{Q}) \ , \] of  quasi-isomorphism  that are morphisms of differential graded (dg) associative algebras inducing isomorphisms in cohomology. A space $X$ is \emph{formal} when the singular cochain algebra $C^*(X; \mathbb{Q})$ is homotopy equivalent to its cohomology algebra. In this case, any invariant of its rational homotopy type can be computed from its cohomology ring. \medskip

\noindent Formality is often imposed by additional geometric structure. Classical examples include compact Kähler manifolds and positive quaternionic--Kähler manifolds, both of which are formal; see \cite{DGMS75,AK12}. In the latter setting, formality was studied as a first step toward the LeBrun--Salamon conjecture, which predicts that every positive quaternionic--Kähler manifold is a symmetric space. Since symmetric spaces are formal, any failure of formality would have immediately obstructed this conjecture. Related, though weaker, phenomena occur in Sasakian geometry. Sasakian manifolds have vanishing higher-order Massey products; see \cite{BFVT16}. Altogether, these examples show that geometric structures can place strong constraints on the rational homotopy type of a manifold. Hence, proving formality or an appropriate weaker substitute, can be a powerful tool for deriving concrete geometric consequences for the manifolds under consideration. \medskip 

\noindent
Analogous questions can be considered over a general commutative coefficient ring $R$. In this setting, singular cochains $C^*(X;R)$ are naturally regarded as associative dg algebras or, more precisely, as $E_\infty$-algebras. Mandell's theorem \cite{Man06} shows that, in the appropriate setting, the $E_\infty$-algebra structure on cochains contains remarkably strong information about the homotopy type of $X$. Even the underlying associative dg algebra can retain significant geometric and topological information. For example, for an oriented closed manifold, it contains several invariants dealing with its string topology.

\noindent Homotopy equivalences between singular cochain algebras appear as a particular case of the following much more general notion. \medskip

\paragraph{\bf Definition.} Let $R$ be a commutative ground ring and let $A$ and $B$ be  dg $R$-modules. Let $\P$ be a type of algebraic structure (e.g. Lie algebras, Frobenius bialgebras, operads, etc.) 
\begin{enumerate}
	\item Two dg $\P$-algebra structures $(A,\varphi)$ and $(B,\psi)$ are said to be \emph{homotopy equivalent} if there exists a zig-zag of quasi-isomorphisms of dg $\P$-algebras relating them:
	\[(A,\varphi) \overset{\sim}{\longleftarrow} \cdot \overset{\sim}{\longrightarrow} \cdots \overset{\sim}{\longleftarrow} \cdot \overset{\sim}{\longrightarrow} (B, \psi) \ . \] 
	\item A dg $\P$-algebra $(A, \varphi)$ is said to be \emph{formal} if it is homotopy equivalent to the dg $\P$-algebra induced on its homology, $\left(H(A),\varphi^*\right)$ .
\end{enumerate}

\noindent The purpose of this article is to construct invariants characterizing homotopy equivalence and formality for such algebraic structures. \medskip

\paragraph{\bf Kaledin classes.}
The Kaledin classes constructed in \cite{Kaledin} provide a complete obstruction to the formality of algebraic structures encoded by a colored operad or a properad $\P$. Their construction rests on the following fact: in characteristic zero, homotopy-equivalence problems can be recast as deformation problems, \cite{graphexp}. More precisely, after transferring the algebraic structure to homology, homotopy equivalences correspond to gauge equivalences between Maurer--Cartan elements in the dg Lie algebra
\[\mathfrak{g} \coloneqq \Hom_{\mathbb{S}} \left(\Bar{\P}, \End_{H(A)}\right) \ , \]
see Section~\ref{2.1}. From this perspective, formality is the special case of gauge triviality: the dg $\P$-algebra $(A,\varphi)$ is formal precisely when the Maurer--Cartan element encoding its transferred structure is gauge equivalent to zero. The Kaledin class $K_{\varphi}$ detects exactly this phenomenon: it vanishes if and only if $(A,\varphi)$ is formal.

\begin{center}
	\emph{Can Kaledin classes be extended from formality to homotopy equivalence?}
\end{center}

\noindent
Equivalently, can one detect whether two arbitrary Maurer--Cartan elements in a complete dg Lie algebra are gauge equivalent? The main purpose of the present article is to answer this question.
\medskip

\paragraph{\bf Obstruction sequences to gauge equivalences}
Let $\varphi$ and $\psi$ be Maurer--Cartan elements in a complete dg Lie algebra $(\mathfrak g,d)$. Twisting by $\psi$ turns the problem into a gauge-triviality problem: the difference $\varphi-\psi$ is a Maurer--Cartan element in the twisted dg Lie algebra
\[
\mathfrak g^\psi
\coloneqq
\bigl(\mathfrak g,d+[\psi,-]\bigr),
\]
and $\varphi$ and $\psi$ are gauge equivalent if and only if $\varphi-\psi$ is gauge equivalent to zero in $\mathfrak g^\psi$. Although the Kaledin class is no longer available in this general twisted setting, the problem can still be solved order by order along the complete descending filtration $\mathcal F$ of $\mathfrak g$. Starting from a gauge trivialization modulo $\mathcal F^k\mathfrak g$, we construct a homology class
\[
\vartheta_k
\in
H_{-1}
\left(
\mathfrak g^\psi/\mathcal F^{k+1}\mathfrak g
\right)
\]
whose vanishing is precisely the obstruction to extending this trivialization modulo $\mathcal F^{k+1}\mathfrak g$. The construction is therefore inductive: as long as $\vartheta_k$ vanishes, one may proceed to the next stage; the first nonzero class detects the exact order at which gauge equivalence fails. This produces a \emph{gauge-triviality sequence}
$\left(\vartheta_k\right)_{1\leqslant k\leqslant n}$
of one of the following two types:
\begin{itemize}
	\item[$\centerdot$]
	an infinite sequence of vanishing classes, when $n=\infty$;
	\item[$\centerdot$]
	a finite sequence, when $n<\infty$.
\end{itemize}
The index
$
n\in[\![1,\infty]\!]$
depends only on the pair $(\varphi,\psi)$ and is called its \emph{gauge-equivalence degree}. It measures how far the Maurer--Cartan elements remain gauge equivalent along the filtration. In particular, we have $n=\infty$ if and only if $\varphi$ and $\psi$ are gauge equivalent modulo $\mathcal F^k\mathfrak g$ for every $k$; see Theorem~\ref{suites d'obstru}. Since the construction depends only on the complete dg Lie algebra governing the deformation problem, it applies not only to homotopy equivalences, but more generally to any deformation problem encoded by such an algebra. As a first illustration, Section~\ref{section1} concludes with an application to the homotopy triviality of fibrations.
\medskip

\noindent \textbf{Theorem \ref{fibrations}.}
\textit{Let $X$ be a simply connected topological space and let $F$ be a nilpotent space of finite $\mathbb{Q}$-type. A fibration $\xi$ over $X$, with fiber $F_{\mathbb{Q}}$, is trivial up to homotopy if and only if the fibration obtained by extending  the scalars fiberwise to $\mathbb{R}$, is trivial up to homotopy.} \bigskip

\paragraph{\bf Obstruction sequences to homotopy equivalences.}
In Section~\ref{section2}, we apply gauge triviality sequences to the main deformation problem considered in this article: determining whether two algebraic structures encoded by a properad or a colored operad are homotopy equivalent. Let $R$ be a $\mathbb{Q}$-algebra, let $\C$ be a reduced conilpotent dg coproperad (respectively, a reduced conilpotent colored dg cooperad), and let $\Omega\C$ denote its cobar construction. Given two $\Omega\C$-algebras whose underlying complexes have isomorphic cohomology, their transferred structures may be viewed, after choosing such an isomorphism, as Maurer--Cartan elements in a common complete dg Lie algebra $\mathfrak g$. The gauge-equivalence degree introduced above then measures, order by order along the filtration of $\mathfrak g$, how far these transferred structures are from being gauge equivalent. This yields the following criterion.

\medskip

\noindent \textbf{Theorem~\ref{obstru en char 0}.}
\textit{Let $A$ and $B$ be chain complexes over $R$ with isomorphic cohomology, and let $(A,\varphi)$ and $(B,\psi)$ be two $\Omega\C$-algebras whose structures admit transfer to cohomology. The following assertions are equivalent:}
\begin{enumerate}
	\item \textit{The algebras $(A,\varphi)$ and $(B,\psi)$ are homotopy equivalent modulo $\mathcal F^k\mathfrak g$ for every $k$.}
	\item \textit{The gauge-equivalence degree of their transferred structures is infinite.}
\end{enumerate}

\medskip

\noindent Under suitable boundedness or weight-grading assumptions, these equivalent conditions are further equivalent to the existence of an actual homotopy equivalence between $(A,\varphi)$ and $(B,\psi)$; see Theorems~\ref{the bounded case} and~\ref{the weight case}. In the weight-graded setting, the result extends to arbitrary commutative ground rings, and the successive obstructions can be identified with truncations of the Kaledin class; see Sections~\ref{2.4} and~\ref{2.5}. One first application is a generalization of formality descent (\cite[Section~4.1]{Kaledin}) to homotopy equivalence decent, see Theorem \ref{eq descent}. 

 \medskip

\noindent Theorem~\ref{obstru en char 0} applies, in particular, to algebras encoded by a properad $\P$, thus allowing for operations with multiple outputs. To treat the case of positive characteristic, Section~\ref{2.3} refines the properadic calculus over arbitrary commutative coefficient rings. A notable application of this framework arises from joint work with Alex Takeda \cite{ET24} on the properadic coformality of spheres. Recall that a topological space $X$ is said to be coformal when the chain algebra of its based loop space, $C_*(\Omega X;\mathbb{Q})$, is formal as an $\mathcal{A}_{\infty}$-algebra. For an oriented manifold, however, this algebra carries additional structure: Kontsevich, Takeda, and Vlassopoulos showed in \cite{KTV21} that Poincaré duality can be encoded by equipping $C_*(\Omega X;\mathbb{Q})$ with a pre-Calabi--Yau structure. Properadic coformality asks for more than ordinary coformality: the zig-zag of quasi-isomorphisms realizing formality must preserve this additional structure, and hence Poincaré duality. In \cite{ET24}, we use gauge-triviality sequences to prove a stronger form of properadic coformality for spheres:

\medskip

\noindent \textbf{Theorem~\cite{ET24}.}
	\textit{The pre-Calabi--Yau algebra $C_*(\Omega S^n;\mathbb{Q})$ is intrinsically formal for any $n \geqslant 1$, i.e. determined up to homotopy by the structure induced on homology.}
	
	\medskip

\paragraph{\bf Minimal models for highly connected varieties.}

In \cite{Mil79}, Miller proved that the de Rham algebra of a compact $k$-connected manifold of dimension $d<4k+2$ is formal. This result was extended by Zhou \cite{Zho22}, who showed that, for $\ell\geqslant 3$, the de Rham algebra of a compact $k$-connected manifold $M$ of dimension $d<(\ell+1)k+2$ has depth $\ell - 1$: it is homotopy equivalent to a strictly unital $\mathcal A_\infty$-algebra of the form
\[
\bigl(H_{\mathrm{dR}}^*(M),\psi_2,\ldots,\psi_{\ell-1}\bigr).
\]
In Section~\ref{section3}, we use the obstruction theory developed above to extend these results to other cohomology theories and coefficient rings.

\medskip

\noindent \textbf{Theorem~\ref{poincaré}.}
\textit{Let $\ell\geqslant 3$, and let $\mathbb K$ be a field in which $\ell$ and $\ell+1$ are invertible. Let $M$ be a compact $k$-connected smooth manifold of dimension $d<(\ell+1)k+2$. Then its singular cochain algebra $C_{\mathrm{sing}}^*(M;\mathbb K)$ has depth $\ell-1$.}

\medskip

\noindent An analogous statement holds for étale cochains.

\medskip

\noindent \textbf{Theorem~\ref{poincaré2}.}
\textit{Let $q$ be a prime number, and let $K$ be a separably closed field in which $q$ is invertible. Let $\ell\geqslant 3$ be such that $\ell$ and $\ell+1$ are invertible in $\mathbb F_q$. Let $X$ be a $k$-connected irreducible smooth proper variety over $K$ of dimension $d<(\ell+1)k+2$. Then its étale cochain algebra $C_{\mathrm{\acute{e}t}}^*(X;\mathbb F_q)$ has depth $\ell-1$.}

\medskip

\paragraph{\bf Notations and conventions. }

\begin{itemize}
	\item[$\centerdot$] We work over a commutative ground ring $R$ of characteristic $\mathrm{char}(R)$. All tensor products are taken over $R$ and every morphism is $R$-linear unless otherwise specified.
	\item[$\centerdot$]  We work in the symmetric monoidal category of chain complexes over $R$ (with the Koszul sign rule). If $A$ is a chain complex and $x \in A$ is a homogeneous element, we denote by $|x|$ its homological degree.
	\item[$\centerdot$] The abbreviation ``dg'' stands for the words ``differential graded''.
	\item[$\centerdot$]  We use the notations of \cite{LodayVallette12} for operads and the ones of \cite{PHC} for properads. \medskip
\end{itemize}

\noindent \textbf{Acknowledgements}.
I would like to thank my advisors, Geoffroy Horel and Bruno Vallette, for their constant support. I am grateful to Olivier Benoist, Alexander Berglund, and Jiawei Zhou for enlightening conversations related to \cite{SGA4}, \cite{berglund15}, and \cite{Zho22}, respectively. \medskip

\section{\textcolor{bordeau}{Obstruction sequences to gauge equivalences}}\label{section1}

 In this section, the ground ring $R$ is a $\mathbb{Q}$-algebra. Let $\varphi$ and $\psi$ be two Maurer--Cartan elements in a complete dg Lie algebra $\mathfrak{g}$: are they gauge equivalent? Section \ref{1.2} gives a first answer by constructing an obstruction sequence detecting whether $\varphi$ and $\psi$ are gauge equivalent modulo $\mathcal{F}^n \mathfrak{g}$, for all $n \geqslant 1$, where $\mathcal{F}$ denotes the complete descending filtration. This implies that $\varphi$ and $\psi$ are gauge equivalent in several cases, such as when the dg Lie algebra $\mathfrak{g}$ is bounded (see Section \ref{bound}), weight-graded (see Section \ref{formality}) or when they become gauge equivalent after scalar extension (see Section \ref{descent})

\subsection{Complete dg Lie algebras}\label{1.1}
We start by recalling the deformation theory controlled by a complete dg Lie algebra, see e.g. \cite[Chapter~1]{DSV22} for more details.

\begin{definition}[Complete dg Lie algebra]
	A \emph{complete differential graded (dg) Lie algebra} \[(\mathfrak{g}, [-,-], d, \mathcal{F})\] is a graded $R$-module $\mathfrak{g}$ equipped with 
	
	\begin{itemize}
		\item[$\centerdot$] a Lie bracket $[-,-] : \mathfrak{g} \otimes \mathfrak{g} \to \mathfrak{g}$, i.e. homogeneous map of degree $0$, satisfying the Jacoby identity and the antisymmetry properties; 
		\item[$\centerdot$] a differential $d : \mathfrak{g} \to \mathfrak{g}$, i.e. a degree $-1$ derivation that squares to zero; 
		\item[$\centerdot$] a complete descending filtration $\mathcal{F}$, i.e. a decreasing chain of sub-complexes \[\mathfrak{g} = \mathcal{F}^1 \mathfrak{g} \supset \mathcal{F}^2 \mathfrak{g} \supset \mathcal{F}^3 \mathfrak{g} \supset \cdots  \] compatible with the bracket $\left[\mathcal{F}^n \mathfrak{g}, \mathcal{F}^m \mathfrak{g} \right] \subset \mathcal{F}^{n+m} \mathfrak{g}$, and such that the canonical projections $\pi_n : \mathfrak{g} \twoheadrightarrow  \mathfrak{g} /\mathcal{F}^n \mathfrak{g}$ lead to an isomorphism \[\pi : \mathfrak{g} \to \lim_{n \in \mathbb{N}} \mathfrak{g} / \mathcal{F}^n\mathfrak{g} \ . \] 
	\end{itemize}
	In the sequel, we will usually use the same notation $\mathfrak{g}$ for the underlying chain
	complex and the full data of complete dg Lie algebra.
\end{definition}

\begin{definition}[Maurer--Cartan element]
	The set of \emph{Maurer--Cartan elements} of a complete dg Lie algebra $\mathfrak{g}$ is defined as \[\mathrm{MC}(\mathfrak{g}) \coloneqq \lbrace \varphi \in \mathfrak{g}_{-1} \mid  d\left(\varphi\right) + \tfrac{1}{2}\left[\varphi, \varphi \right] = 0\rbrace \ .\]
\end{definition}

\begin{proposition}
	The \emph{gauge group} of a complete dg Lie algebra $\mathfrak{g}$ is the group $\left(\mathfrak{g}_0, \mathrm{BCH}, 0\right)$ obtained from the set $\mathfrak{g}_0$ via Baker--Campbell--Hausdorff formula defined by \[\mathrm{BCH}(\lambda,\nu)\coloneqq \mathrm{ln}\left(e^{\lambda}e^{\nu}\right) \ , \] in the associative algebra of formal power series on $\lambda$ and $\nu$. It acts on the set of Maurer--Cartan elements through the gauge action defined for $\lambda \in \mathfrak{g}_0$ and $\varphi \in \mathrm{MC}(\mathfrak{g})$ by \[ \lambda \cdot \varphi \coloneqq e^{ \mathrm{ad}_{\lambda}}(\varphi) - \frac{ e^{ \mathrm{ad}_{\lambda}} - \mathrm{id}}{\mathrm{ad}_{\lambda}} (d \lambda) \ ,  \]  where $\mathrm{ad}_{\lambda} \coloneqq \left[\lambda, -\right]$ is the adjoint operator. 
\end{proposition}

\begin{proof}
	See \cite[Theorem~1.53]{DSV22}. 
\end{proof}

\begin{definition}[Gauge equivalences]\label{gauge equivalences}
	Two Maurer--Cartan elements $\varphi$ and $\psi$ in a complete dg Lie algebra $\mathfrak{g}$ are \emph{gauge equivalent} if there exists $\lambda \in \mathfrak{g}_0$ such that \[ \lambda \cdot \varphi = \psi \ .\] The \emph{moduli space of Maurer–Cartan elements} $\mathcal{MC}(\mathfrak{g})$ is the coset of Maurer–Cartan elements modulo the gauge action. A Maurer--Cartan element $\varphi$ is said \emph{gauge trivial} if it is gauge equivalent to zero. 
\end{definition}

\begin{proposition}\label{tech}
	Let $\mathfrak{g}$ be a complete dg Lie algebra and let $\varphi$ and $\psi$ in $\mathrm{MC}(\mathfrak{g})$. 
	\begin{enumerate}
		\item The operator $d^{\psi} \coloneqq d + \mathrm{ad}_{\psi}$ is a differential of $\mathfrak{g}$ and defines a complete dg Lie algebra \[\mathfrak{g}^{\psi} \coloneqq \left(\mathfrak{g},  [-,-], d^{\psi}, \mathcal{F} \right) \ .\]  
		\item The difference $\varphi - \psi$ is a Maurer-Cartan element in $\mathfrak{g}^{\psi}$. 
		\item For all $\lambda \in \mathfrak{g}_0$, the following relations are equivalent: 
		\[ \lambda \cdot \varphi  = \psi  \; \mbox{in} \; \mathfrak{g} \; \Longleftrightarrow \; \lambda \cdot (\varphi - \psi) = 0 \; \mbox{in} \; \mathfrak{g}^{\psi}  \ .\] 
	\end{enumerate}
\end{proposition}

\begin{proof} 
For Point (1) and (2), we refer the reader to \cite[Proposition~1.43, Lemma~1.56]{DSV22}. Point (3) follows from the following equivalences:   \[	\begin{split}                                  
	\lambda \cdot \varphi  = \psi  \; \mbox{in} \; \mathfrak{g} 
	& \; \Longleftrightarrow \;  e^{ \mathrm{ad}_{\lambda}}(\varphi) - \frac{ e^{ \mathrm{ad}_{\lambda}} - \mathrm{id}}{\mathrm{ad}_{\lambda}} (d \lambda) = \psi \\
	& \; \Longleftrightarrow  \; e^{ \mathrm{ad}_{\lambda}}(\varphi - \psi) - \frac{ e^{ \mathrm{ad}_{\lambda}} - \mathrm{id}}{\mathrm{ad}_{\lambda}} (d^{\psi} (\lambda)) = 0  \\
	&\; \Longleftrightarrow \; \lambda \cdot (\varphi - \psi) = 0 \; \mbox{in} \; \mathfrak{g}^{\psi} \ . 
\end{split}  \]
\end{proof}

\subsection{Obstruction sequences to gauge equivalences}\label{1.2}

Let $\varphi$ and $\psi$ be two Maurer--Cartan elements in a complete dg Lie algebra $\mathfrak{g}$. The key issue is to detect whether $\varphi$ and $\psi$ are gauge equivalent. This section gives a first answer by constructing an obstruction sequence detecting whether for all $n \geqslant 1$, there exists $\omega_n \in \mathfrak{g}_0$, such that \[\omega_n \cdot \varphi \equiv \psi \pmod{\mathcal{F}^n \mathfrak{g}} \  .\] Let us set $\mathfrak{h} \coloneqq \mathfrak{g}^{\psi}$ and $\phi \coloneqq \varphi - \psi  \in \mathrm{MC}(\mathfrak{h})$. By Point (3) of Proposition \ref{tech}, we have \[\omega_n \cdot \varphi \equiv \psi \pmod{\mathcal{F}^n \mathfrak{g}} \iff \omega_n \cdot \phi \in \mathcal{F}^n \mathfrak{h} \ . \]

\begin{proposition}\label{obstru}
	Let $\mathfrak{h}$ be a complete dg Lie algebra and let $n \geqslant 1$. Let $\phi \in \mathrm{MC}(\mathfrak{h})$ be a Maurer--Cartan element such that $\phi \in \mathcal{F}^n \mathfrak{h}$. Let us consider \[\vartheta_{n} \coloneqq \left[\pi_{n+1} (\phi) \right] \in H_{-1}\left( \mathfrak{h} / \mathcal{F}^{n+1}\mathfrak{h}  \right) \ .\] The following assertions are equivalent.
	\begin{enumerate}
		\item The homology class $\vartheta_{n}$ vanishes. 	
		\item There exists $\upsilon \in \mathfrak{h}_0$ such that $\upsilon \cdot \phi \in  \mathcal{F}^{n+1} \mathfrak{h}$~.
	\end{enumerate}
\end{proposition}

\begin{proof}
	For all $\upsilon \in \mathfrak{h}_0$, the gauge action formula gives \begin{equation}\label{eq}
	 \upsilon \cdot \phi \equiv  \phi - \frac{ e^{ \mathrm{ad}_{\upsilon}} - \mathrm{id}}{\mathrm{ad}_{\upsilon}} (d \upsilon)  \pmod{\mathcal{F}^{n+1} \mathfrak{h}} \ .
\end{equation}  If $\vartheta_{n} = 0$, then there exists $\upsilon \in \mathfrak{h}_0$ such that \[ \phi \equiv d \upsilon \pmod{\mathcal{F}^{n+1} \mathfrak{h}} \ .\] Since $\phi \in \mathcal{F}^n \mathfrak{h}$, this implies that $d \upsilon \in \mathcal{F}^n \mathfrak{h}$ and \[ \frac{ e^{ \mathrm{ad}_{\upsilon}} - \mathrm{id}}{\mathrm{ad}_{\upsilon}} (d \upsilon) \equiv d \upsilon \pmod{\mathcal{F}^{n+1} \mathfrak{h}} \ . \] Equation (\ref{eq}) then implies that $\upsilon \cdot \phi \in \mathcal{F}^{n+1} \mathfrak{h} \ .$ Conversely, if point (2) holds, we have \[ \phi  \equiv \frac{ e^{ \mathrm{ad}_{\upsilon}} - \mathrm{id}}{\mathrm{ad}_{\upsilon}} (d \upsilon)   \pmod{\mathcal{F}^{n+1} \mathfrak{h}}  \]  by Equation (\ref{eq}). Since $\phi \in \mathcal{F}^n \mathfrak{h}$, one can prove by induction on $k$, for all $1 \leqslant k \leqslant n $, that   \[d \upsilon \in  \mathcal{F}^{k} \mathfrak{h}  \quad \mbox{and} \quad \phi  \equiv d \upsilon  \pmod{\mathcal{F}^{n+1} \mathfrak{h}} \ . \] This implies that $\vartheta_{n} = 0$. 
\end{proof}

\begin{construction} \label{constru}
Let $\phi \in \mathrm{MC}(\mathfrak{h})$ be a Maurer--Cartan element in a complete dg Lie algebra $\mathfrak{h}$. We aim to detect whether $\phi$ is gauge trivial. Let us set $\phi_1 \coloneqq \phi$ and let us consider \[\vartheta_1 \coloneqq \left[ \pi_2 (\phi_1)\right] \in H_{-1}\left( \mathfrak{h} / \mathcal{F}^{2} \mathfrak{h}  \right) . \]
	\begin{itemize}
		\item[$\centerdot$] If $\vartheta_1 \neq 0$~, then $\phi$ is not gauge trivial, by the implication $(2) \Rightarrow (1)$ of Proposition~\ref{obstru}. 
		\item[$\centerdot$] If $\vartheta_1 = 0$~, there exists $\upsilon_1 \in \mathfrak{h}_0$~, such that $\upsilon_1 \cdot \phi_1 \in  \mathcal{F}^{2} \mathfrak{h}$~, by the implication $(1) \Rightarrow (2)$ of Proposition \ref{obstru}.
	\end{itemize}
	If $\vartheta_1 = 0$~, we set $\phi_2 \coloneqq \upsilon_1 \cdot \phi$ and \[\vartheta_2 \coloneqq \left[ \pi_3 (\phi_2)\right] \in H_{-1}\left( \mathfrak{h} / \mathcal{F}^{3} \mathfrak{h}  \right) . \]
	\begin{itemize}
		\item[$\centerdot$] If $\vartheta_2 \neq 0$~, then $\phi$ is not gauge trivial. Indeed, if there exists $\lambda \in \mathfrak{h}_0$ such that $\lambda \cdot \phi = 0$ then $\mathrm{BCH}(\lambda, - \upsilon_1) \cdot \phi_2 = 0 $ and $\vartheta_2 = 0$~, by Proposition \ref{obstru}. 
		\item[$\centerdot$] If $\vartheta_2 = 0$~, there exists $\upsilon_2 \in \mathfrak{h}_0$~, such that $\upsilon_2 \cdot \phi_2  \in \mathcal{F}^{3} \mathfrak{g}$~, by Proposition \ref{obstru}.
	\end{itemize}
	If $\vartheta_2 = 0$~, let us set $\phi_3 \coloneqq \upsilon_2 \cdot \phi_2$ and \[\vartheta_3 \coloneqq \left[ \pi_4 (\phi_3)\right] \in H_{-1}\left( \mathfrak{h} / \mathcal{F}^{4} \mathfrak{h}  \right) . \] Once again, the following assertions hold by Proposition \ref{obstru}. 
	\begin{itemize}
		\item[$\centerdot$] If $\vartheta_3 \neq 0$~, then $\phi$ is not gauge trivial.   
		\item[$\centerdot$] If $\vartheta_3 = 0$~, there exists $\upsilon_3 \in \mathfrak{h}_0$~, such that $\upsilon_3 \cdot \phi_3  \in \mathcal{F}^{4} \mathfrak{h}$~.
	\end{itemize}

	\noindent The construction of such obstruction classes can be performed higher up in a similar way. This leads to a sequence of classes $\left(\vartheta_k \right)_{1 \leqslant k \leqslant n}$ which is either 
	\begin{itemize}
		\item[$\centerdot$] an infinite sequence of vanishing homology classes, when $n = \infty$~, or
		\item[$\centerdot$]  a finite sequence of trivial classes that ends on a nonzero class $\vartheta_{n}$~, when $n \in \mathbb{N}$~.
	\end{itemize}  Any such sequence is not unique and depends on the choice of gauge $\upsilon_k$ made at each level. 
\end{construction}

\begin{definition}[Gauge triviality sequence]\label{gauge triviality sequ}
	A \emph{gauge triviality sequence} of a Maurer--Cartan element $\phi \in \mathrm{MC}(\mathfrak{h})$ is an obstruction sequence $\left(\vartheta_k \right)_{1 \leqslant k \leqslant n}$~, for $n\in [\![1, \infty]\!]$~, obtained through Construction \ref{constru}. 
\end{definition}

\begin{lemma}\label{gauge deg} Let $\mathfrak{h}$ be a complete dg Lie algebra and let $\phi \in \mathrm{MC}(\mathfrak{h})$ be a Maurer--Cartan element. Let $\left(\vartheta_k \right)_{1 \leqslant k \leqslant n}$ be a gauge triviality sequence associated to $\phi$~, with $n\in [\![1, \infty]\!]$~. Any other gauge triviality sequence $\left(\vartheta_k' \right)_{1 \leqslant k \leqslant m}$ satisfies $m=n$~. 
\end{lemma}

\begin{proof} 
We begin this proof by fixing notations. Let us denote by $\left(\upsilon_k \right)$ and $\left(\phi_k \right)$ the sequences of gauges and Maurer--Cartan elements associated to the gauge triviality sequence $\left(\vartheta_k \right)_{1 \leqslant k \leqslant n}$, by Construction \ref{constru}. We have \[\phi_{k+1} = \upsilon_{k} \cdot \phi_{k} \in \mathcal{F}^{k+1} \mathfrak{h} \ , \] for all $1 \leqslant k \leqslant n-1$, with $\phi_1 \coloneqq \phi$~. For all $1 \leqslant k \leqslant n$, we have  \[\vartheta_k = \left[ \pi_{k+1} (\phi_k)\right] \in H_{-1}\left( \mathfrak{h} / \mathcal{F}^{k+1} \mathfrak{h}  \right) \ . \] Let us set $\omega_k \coloneqq \mathrm{BCH}(\upsilon_{k-1}, \mathrm{BCH}(\cdots \mathrm{BCH}(\upsilon_{2}, \upsilon_{1})) \cdots ) \ , $ so that $\phi_k = \omega_k \cdot \phi$, for all $2 \leqslant k \leqslant n$. Suppose that there exists another gauge triviality sequence \[\left(\vartheta_k' \right)_{1 \leqslant k \leqslant m} \ . \] Let us denote by $\left(\tilde{\upsilon}_k \right)$ and  $\left(\psi_k \right)$ the associated sequence of gauges and Maurer--Cartan elements. Let us set $\tilde{\omega}_k \coloneqq \mathrm{BCH}(\tilde{\upsilon}_{k-1}, \mathrm{BCH}(\cdots \mathrm{BCH}(\tilde{\upsilon}_{2}, \tilde{\upsilon}_{1})) \cdots )  \ , $  such that $\psi_{k} = \tilde{\omega}_k \cdot \phi $ for all $2 \leqslant k \leqslant m$. By construction, we have \[\vartheta_k' = \left[ \pi_{k+1} (\psi_k)\right] \in H_{-1}\left( \mathfrak{h} / \mathcal{F}^{k+1} \mathfrak{h}  \right) \ . \] Let us suppose that $\left(\vartheta_k \right)$ is an infinite sequence of vanishing homology classes, i.e. we have $n = \infty$. Suppose that $m$ is finite, so that $\vartheta_m' \neq 0$. The gauge \[\lambda \coloneqq \mathrm{BCH}( \omega_{m+1} , -\tilde{\omega}_m) \] satisfies $\lambda \cdot \psi_m  \in \mathcal{F}^{m+1} \mathfrak{h} $~. By the implication $(2) \Rightarrow (1)$ of Proposition \ref{obstru}, applied to $\psi_m$ and $m$~, this implies that $\vartheta_m' = 0$~. This leads to a contradiction and thus $m = n = \infty$~. Let us now suppose that $n$ is a positive integer, so that $\left(\vartheta_k \right)$ is a finite sequence that ends with a non-trivial class $\vartheta_{n}$~. If $m = \infty$~, it would be inferred that $n$ is also equal to $\infty$ by the previous case. Thus, the sequence $\left(\vartheta_k' \right)$ is finite and it abuts on a non-trivial class $\vartheta_{m}'$~. Without loss of generality, we suppose that $n \leqslant m$~. If  $n < m$~, the gauge \[\lambda \coloneqq \mathrm{BCH}( \tilde{\omega}_{n+1} , - \omega_n )\] is such that $\lambda \cdot \phi_n \in \mathcal{F}^{n+1} \mathfrak{h}$~. By the implication $(2) \Rightarrow (1)$ of Proposition \ref{obstru} applied to $\phi_n$ and $n$~, the class $\vartheta_{n}$ vanishes. This leads to a contradiction and thus $m = n$~.
\end{proof}

\begin{definition}[Gauge equivalence degree]
Let $\varphi$ and $\psi$ be two Maurer--Cartan elements in $\mathfrak{g}$. Their \emph{gauge equivalence degree} is the index $n \in [\![1, \infty]\!]$ of the last class of a gauge triviality sequence of $\varphi - \psi $ in $\mathfrak{g}^{\psi}$. 
\end{definition}

\begin{theorem}\label{suites d'obstru}
Let $\varphi$ and $\psi$ be Maurer--Cartan elements in a complete dg Lie algebra $\mathfrak{g}$. 
	\begin{enumerate}
		\item[(I)] The following assertions are equivalent. 
		\begin{enumerate}
			\item[(1)] The gauge equivalence degree of  $\varphi$ and $\psi$ is equal to $\infty$~.
			\item[(2)] For all $k \geqslant 1$, there exists $\omega_k \in \mathfrak{g}_0$~, such that \[\omega_k \cdot \varphi \equiv \psi \pmod{\mathcal{F}^k \mathfrak{g}} \ . \]
		\end{enumerate}   
		\item[(II)] The following assertions are equivalent. 
		\begin{enumerate}
			\item[(3)] The gauge equivalence degree of  $\varphi$ and $\psi$ is equal to $n \in \mathbb{N}$~.
			\item[(4)]  There exists $\omega_n \in \mathfrak{g}_0$ such that $\omega_n  \cdot \varphi \equiv \psi \pmod{\mathcal{F}^{n} \mathfrak{g}}$ and, for all $\nu \in \mathfrak{g}_0$~,  \[\nu \cdot \varphi \not\equiv \psi \pmod{\mathcal{F}^{n+1} \mathfrak{g}} \ . \] 
		\end{enumerate}		 
	\end{enumerate}
\end{theorem}

\begin{proof} We begin this proof by fixing the notations. Let $\left(\vartheta_k \right)_{1 \leqslant k \leqslant n}$ be a gauge triviality sequence of $\phi\coloneqq \varphi- \psi$ and let us denote by $\left(\upsilon_k \right)$ and $\left(\phi_k \right)$ the associated sequence of gauges and Maurer--Cartan elements given by Construction \ref{constru}. Let us set \[\omega_k \coloneqq \mathrm{BCH}(\upsilon_{k-1}, \mathrm{BCH}(\cdots \mathrm{BCH}(\upsilon_{2}, \upsilon_{1})) \cdots ) \ , \] for all $2 \leqslant k \leqslant n$~, so that, $\phi_k = \omega_k \cdot \phi$~. Let us prove Point (I).  If $n = \infty$~, we have \[\omega_k \cdot \varphi \equiv \psi \pmod{\mathcal{F}^k \mathfrak{g}} \ , \] for all $k \geq 1$, since $\phi_k = \omega_k \cdot \phi \in \mathcal{F}^k \mathfrak{g}$. Conversely, suppose that $n \in \mathbb{N}$ is finite so that $\vartheta_n \neq 0$. Let $\lambda \in \mathfrak{g}_0$ be such that \[\lambda \cdot \varphi \equiv \psi \pmod{\mathcal{F}^{n+1} \mathfrak{g}} \ . \] The gauge $\nu \coloneqq \mathrm{BCH}( \lambda , - \omega_n)$ satisfies $\nu \cdot \phi_n  \in \mathcal{F}^{n+1} \mathfrak{h} \ .$ By the implication $(2) \Rightarrow (1)$ of Proposition \ref{obstru}, applied to $\phi_n$ and $n$~, this implies that $\vartheta_n = 0$ and leads to a contradiction.

	\noindent Let us prove $(3) \Rightarrow (4)$. If $n$ is finite, then $\vartheta_{n-1} = 0$ and $\vartheta_{n} \neq 0$~. By construction, we have \[\omega_{n} \cdot \varphi \equiv \psi \pmod{\mathcal{F}^{n} \mathfrak{g}} \ .\] However, there is no $\nu \in \mathfrak{g}_0$ such that $\nu \cdot \varphi \equiv \psi \pmod{\mathcal{F}^{n+1} \mathfrak{g}}$~. Otherwise, we would have \[\mathrm{BCH}( \nu , - \omega_n) \cdot \phi_n \in \mathcal{F}^{n+1} \mathfrak{h} \ ,\] which implies $\vartheta_n = 0$, by the implication $(2) \Rightarrow (1)$ of Proposition \ref{obstru}, applied to $\phi_n$ and $n$~. Conversely, let us prove $(4) \Rightarrow (3)$. We suppose that there exists $m \in \mathbb{N}$ such that 
	\begin{itemize}
		\item[$\centerdot$] there exists $\lambda \in \mathfrak{g}_0$ such that $\lambda \cdot \varphi \equiv \psi \pmod{\mathcal{F}^{m+1} \mathfrak{g}}$~,
		\item[$\centerdot$] for all $\nu \in \mathfrak{g}_0$~, we have $\nu \cdot \varphi \not\equiv \psi \pmod{\mathcal{F}^{m+1} \mathfrak{g}} \ . $ 
	\end{itemize}
	Let us prove that $m$ indices the last class of the gauge triviality sequence. If $n < m$~, we have \[\mathrm{BCH}( \lambda , - \omega_n) \cdot \phi_n \in \mathcal{F}^{n+1} \mathfrak{h} \ . \] This implies that $\vartheta_n = 0$~, by the implication $(2) \Rightarrow (1)$ of Proposition \ref{obstru}, applied to $\phi_n$ and $n$~. This is a contradiction and thus $m \leqslant n$~. 
	If $m < n$~, then $\vartheta_{m} = 0$ and there exists $\nu \in \mathfrak{h}_0$~, such that $\nu \cdot \phi \in \mathcal{F}^{m+1} \mathfrak{h} \ , $ by the implication $(1) \Rightarrow (2)$ of Proposition \ref{obstru}, applied to $\phi$ and $m$~, which contradicts the second hypothesis.
\end{proof}

\begin{remark}
Under the hypotheses of Theorem \ref{suites d'obstru}, if the gauge equivalence degree is equal to $\infty$, one would expect $\varphi$ and $\psi$ to be gauge equivalent. However, this might not be the case. A good candidate for such a gauge is the infinite composite \begin{equation}\label{composition}
\cdots \mathrm{BCH}(\upsilon_{4}, \mathrm{BCH}(\upsilon_{3}, \mathrm{BCH}(\upsilon_{2}, \upsilon_{1}))) \cdots ) \ , 
\end{equation}  of the gauges associated to a gauge triviality sequence but nothing asserts that this is well defined. In the following sections, we highlight sufficient conditions on the complete dg Lie algebra ensuring the gauge equivalence between $\varphi$ and $\psi$.
\end{remark}

\subsection{The bounded case}\label{bound}
A complete dg Lie algebra $\mathfrak{g}$ is \emph{bounded in degree $-1$} if  there exists $\eta \geqslant 1$ such
that $\mathcal{F}^{\eta} \mathfrak{g}_{-1} = 0 $. In this context, two Maurer--Cartan elements that are gauge equivalent modulo $\mathcal{F}^{\eta} \mathfrak{g}$ are gauge equivalent. This leads to the following improvement of Theorem \ref{suites d'obstru}.

\begin{theorem}[The bounded case]\label{suites d'obstructions3}
	Let $\varphi$ and $\psi$ be two Maurer--Cartan elements in a  dg Lie algebra $\mathfrak{g}$ which is bounded in degree $-1$.  The following assertions are equivalent. 
	\begin{enumerate}
		\item[(1)] The gauge equivalence degree of $\varphi$ and $\psi$ is equal to $\infty$~.
		\item[(2)] The Maurer-Cartan elements $\varphi$ and $\psi$ are gauge equivalent. 
	\end{enumerate}   
\end{theorem}

\begin{proof} 
	Let $\left(\vartheta_k \right)_{1 \leqslant k \leqslant n}$ be a gauge triviality sequence of $\phi \coloneqq \varphi- \psi$ and let us denote by $\left(\upsilon_k \right)$ and $\left(\phi_k \right)$ the associated sequence of gauges and Maurer--Cartan elements given by Construction \ref{constru}.  If $n = \infty$~, it follows from the construction that the gauge  \[\omega_{\eta} \coloneqq \mathrm{BCH}(\upsilon_{\eta}, \mathrm{BCH}(\cdots \mathrm{BCH}(\upsilon_{2}, \upsilon_{1})) \cdots ) \ , \] satisfies $\omega_{\eta} \cdot \varphi = \psi$~. Conversely, suppose that there exists $\lambda$ such that $\lambda \cdot \varphi = \psi$. By the implication (2) $ \Rightarrow$ (1) of Theorem \ref{suites d'obstru}, the gauge equivalence degree is equal to $\infty$. 
\end{proof}

\begin{corollary}
Let $\mathfrak{g}$ be a bounded dg Lie algebra. If there exists $\psi \in \mathrm{MC}(\mathfrak{g})$ such that \[H_{-1}(\mathfrak{g}^{\psi}) = 0 \] then the gauge group action is transitive on $\mathfrak{g}$.
\end{corollary}

\begin{proof}
For any Maurer--Cartan element $\varphi \in \mathrm{MC}(\mathfrak{g})$, the gauge triviality sequence associated to $\phi \coloneqq \varphi - \psi$ in $H_{-1}(\mathfrak{g}^{\psi})$ is an infinite sequence of vanishing homology classes. The gauge equivalence is then equal to $\infty$. By Theorem \ref{suites d'obstructions3}, the Maurer-Cartan element $\varphi$ and $\psi$ are gauge equivalent and the gauge group action has exactly one orbit.
\end{proof}

\subsection{The weight-graded case}\label{formality}

When the complete filtration of $\mathfrak{g}$ comes from an additional weight grading, the constructions of Section \ref{1.2} can be refined in order to faithfully detect gauge equivalences. Let us fix an integer $\delta \geqslant 0$. 

\begin{definition}[$\delta$-weight-graded dg Lie algebra]\label{delta wg}
	A \emph{$\delta$-weight-graded dg Lie algebra} is a complete dg Lie algebra $\left(\mathfrak{g}, [-,-] , d , \mathcal{F} \right)$ with an additional weight grading such that \[\mathfrak{g} \cong \prod_{k \geqslant 1} \mathfrak{g}^{(k)},  \quad \left[ \mathfrak{g}^{(k)},  \mathfrak{g}^{(l)} \right] \subset  \mathfrak{g}^{(k +l)} , \quad \mathcal{F}^n \mathfrak{g} \coloneqq  \prod_{k \geqslant n} \mathfrak{g}^{(k)}  \ ,  \] and a differential satisfying $d = d_0 + d_{\delta}$, \[d_0 \left(\mathfrak{g}^{(k)} \right) \subset \mathfrak{g}^{(k)}\quad \mbox{and} \quad d_{\delta} \left(\mathfrak{g}^{(k)} \right) \subset \mathfrak{g}^{(k + \delta)} \ .\]  Every element $\varphi$ in $\mathfrak{g}$ decomposes as $\varphi = \varphi^{(1)} + \varphi^{(2)}  + \cdots $~, where $\varphi^{(k)} \in \mathfrak{g}^{(k)}$~, for all $k \geqslant 1$~.
\end{definition}

\noindent In this context, one gets the following refinement of Proposition \ref{obstru}.

\begin{proposition}\label{obstructions}
Let $\mathfrak{h}$ be a $\delta$-weight-graded dg Lie algebra and let $n > \delta$ be an integer. Let $\phi \in \mathrm{MC}(\mathfrak{h})$ be such that $\phi \in \mathcal{F}^n \mathfrak{h}$~. Let us consider \[\vartheta_{n} = \left[\phi^{(n)} \right] \in H_{-1}\left( \mathfrak{h} / \mathcal{F}^{n+1}\mathfrak{h}  \right) \ .\] The following assertions are equivalent.
	\begin{enumerate}
		\item The homology class $\vartheta_{n}$ vanishes. 
		\item There exists $\upsilon \in \mathfrak{h}_0$ under the form $\upsilon =  \upsilon^{(n- \delta)} + \upsilon^{(n)} $ such that  $\upsilon \cdot \phi \in \mathcal{F}^{n+1} \mathfrak{h}$~. 		
	\end{enumerate}
\end{proposition}

\begin{proof} 	The implication $(2) \Rightarrow  (1)$ holds, by Proposition \ref{obstru}. Conversely, if $\vartheta_{n} = 0$, there exists $\upsilon \in \mathfrak{h}_0$ such that $\upsilon \cdot \phi  \in \mathcal{F}^{n+1} \mathfrak{h}$~, by Proposition \ref{obstru}. By definition, we have \[\upsilon \cdot \phi \equiv \phi^{(n)} - \sum_{l \geqslant 1} \tfrac{1}{l!} \mathrm{ad}_{\upsilon}^{l-1} \left( d \upsilon   \right) \pmod{\mathcal{F}^{n+1} \mathfrak{g}} \ .    \] For $n \geqslant \delta + 1$, one can prove by induction on $k$ that for all $1 \leq k < n$, \[d \left(\upsilon\right)^{(k)} = 0  \quad \mbox{and} \quad d \left(\upsilon\right)^{(n)} =  \phi^{(n)} \ . \] Thus, the gauge $ \upsilon^{(n - \delta)} + \upsilon^{(n)}$ satisfies the desired property.
\end{proof}

\begin{theorem}[The weight-graded case]\label{suites d'obstructions}
Let $\varphi$ and $\psi$ be two Maurer--Cartan elements in a $\delta$-weight-graded dg Lie algebra $\mathfrak{g}$. Suppose that $\psi \in \mathfrak{g}^{(\delta)}$ so that $\mathfrak{h} \coloneqq \mathfrak{g}^{\psi}$ is also a $\delta$-weight-graded dg Lie algebra. Suppose that $\phi \coloneqq \varphi - \psi \in \mathrm{MC}(\mathfrak{h})$ is in $\mathcal{F}^{\delta +1} \mathfrak{h}$. The following assertions are equivalent. 
	\begin{enumerate}
		\item[(1)] The gauge equivalence degree of $\varphi$ and $\psi$ is equal to $\infty$~.
		\item[(2)] The Maurer-Cartan elements $\varphi$ and $\psi$ are gauge equivalent. 
	\end{enumerate}   
\end{theorem}

\begin{proof} 
The implication (2) $ \Rightarrow$ (1) follows from Theorem \ref{suites d'obstru}. Conversely, suppose that the gauge equivalence degree is equal to $\infty$. Let $\left(\vartheta_k \right)$ be a gauge triviality sequence of $\phi$ and let us denote by $\left(\upsilon_k \right)$ and $\left(\phi_k \right)$ the associated sequence of gauges and Maurer--Cartan elements given by Construction \ref{constru}. Since $\phi \in \mathcal{F}^{\delta +1} \mathfrak{h}$, we can assume that $\phi_{k + 1}  = \phi$ and that $\upsilon_k$ is the unit element for all $1 \leqslant k \leqslant \delta$. By Proposition \ref{obstructions}, we can further assume that $\upsilon_{k} \in  \mathfrak{h}^{(k - \delta)} \oplus \mathfrak{h}^{(k)}$, for all $k \geqslant \delta +1$. If $n = \infty$, it follows from the construction that   \[\omega\coloneqq \mathrm{BCH}(\cdots \mathrm{BCH}(\upsilon_{\delta +3}, \mathrm{BCH}(\upsilon_{\delta +2}, \upsilon_{\delta +1})) \cdots ) \ , \] is well defined and satisfies $\omega\cdot \varphi = \psi$~.  
\end{proof}

\begin{corollary}
	Let $\mathfrak{g}$ be a $\delta$-weight-graded dg Lie algebra. Let $\psi \in \mathrm{MC}\left(\mathfrak{g}^{(\delta)}\right)$ be such that \[\mathcal{F}^{\delta +1} H_{-1}(\mathfrak{g}^{\psi}) = 0 \ .\] Then, any Maurer-Cartan element $\varphi$ such that $\varphi - \psi \in \mathcal{F}^{\delta +1} \mathfrak{g}$ is gauge equivalent to $\psi$. 
\end{corollary}

\begin{proof}
For any Maurer--Cartan element $\varphi \in \mathrm{MC}(\mathfrak{g})$ such that $\varphi - \psi \in \mathcal{F}^{\delta +1} \mathfrak{g}$, let $(\vartheta_k)$ be a gauge triviality sequence associated to $\phi \coloneqq \varphi - \psi$ in $H_{-1}(\mathfrak{g}^{\psi})$. Since $\phi \in \mathcal{F}^{\delta +1} \mathfrak{h}$, we can assume that $\phi_{k + 1}  = \phi$, $\upsilon_k$ is the unit element and $\vartheta_k =0$, for all $1 \leqslant k \leqslant \delta$. Furthermore for all $k \geqslant \delta + 1$, we have \[\vartheta_k \in \mathcal{F}^{\delta +1} H_{-1}\left(\mathfrak{g}^{\psi} / \mathcal{F}^{k+1} \mathfrak{g}^{\psi}\right) = 0  \ . \]  The gauge triviality sequence is then an infinite sequence of vanishing homology classes and th gauge equivalence degree of $\varphi$ and $\psi$ is equal to $\infty$. By Theorem \ref{suites d'obstructions}, the Maurer-Cartan element $\varphi$ and $\psi$ are gauge equivalent.
\end{proof}

\subsection{Equivalence descent}\label{descent}

 Let $R \to S$ be a morphism. If two Maurer--Cartan elements $\varphi$ and $\psi$ are gauge equivalent in $\mathfrak{g}$, then $\varphi \otimes 1$ and $\psi \otimes 1$ are gauge equivalent in $\mathfrak{g} \otimes S$. In this section, we prove that the converse holds true if $S$ is a faithfully flat $R$-algebra, by using the obstruction theory of Section \ref{1.2}.

\begin{theorem}[Equivalence descent] \label{cas complet}
	Let $S$ be a faithfully flat commutative $R$-algebra. Let $\mathfrak{g}$ be a complete dg Lie algebra over $R$ and let $\varphi$ and $\psi$ be two Maurer--Cartan elements in $\mathfrak{g}\ .$ Then $\varphi$ and $\psi$ are gauge equivalent in $\mathfrak{g}$ if and only if $\varphi \otimes 1$ is gauge equivalent to $\psi \otimes 1$ in $ \mathfrak{g} \otimes S$~. 
\end{theorem}

\begin{proof}
	By Proposition \ref{tech}, it suffices to prove that $\phi  = \varphi - \psi$ is gauge trivial in $\mathfrak{h} = \mathfrak{g}^{\psi}$ if and only if $\phi \otimes 1$ is gauge trivial in $\mathfrak{h}\otimes S$. The direct implication is immediate. Conversely, suppose that there exists $\lambda \in \mathfrak{h}$ such that \begin{equation}\label{eq9}
		\lambda \cdot  \left(\phi \otimes 1\right)  = 0 \quad \mbox{in} \quad \mathfrak{h} \otimes S \ .
	\end{equation}   Let us construct a gauge triviality sequence of $\phi$ as follows. By (\ref{eq9}), we have \[\phi \otimes 1 \equiv d^{\psi \otimes 1} \left(\lambda\right) \pmod{\mathcal{F}^{2} \mathfrak{h} \otimes S}  \ .  \] Let us note that the scalar extension $\mathfrak{h} \to \mathfrak{h} \otimes S,  x \mapsto x \otimes 1$ is injective since $R \to S$ is faithfully flat. Picking a $R$-module retract for the inclusion $R \subset S$ induces a retraction $r : \mathfrak{h} \otimes S \to \mathfrak{h}$. The above equivalence leads to  \[\phi  \equiv d^{\psi } (r(\lambda)) \pmod{\mathcal{F}^{2} \mathfrak{h}} \quad \mbox{and} \quad \vartheta_1 = \left[ \pi_2 (\phi)\right] = 0\ . \] Let us set $\upsilon_1 = r(\lambda)$ and $\phi_2 = \upsilon_1 \cdot \phi$. The element $\lambda_2 \coloneqq \mathrm{BCH}(\lambda, - \upsilon_1 \otimes 1)$ defines a gauge equivalence between $\phi_2 \otimes 1$ and $0$. As before, it leads to \[\phi_2  \equiv d^{\psi} (r(\lambda_2)) \pmod{\mathcal{F}^{2} \mathfrak{h}} \quad \mbox{and} \quad \vartheta_2 = \left[ \pi_3 (\phi_2)\right] =0 \ . \] Let us set $\upsilon_2 = r(\lambda_2)$ and $\phi_3 = \upsilon_2 \cdot \phi_2 .$  The construction can be performed higher up. This leads to a gauge triviality sequence $(\vartheta_k)$ which is an infinite sequence of vanishing homology classes associated with a sequence a Maurer--Cartan element $(\phi_k)$ and gauges $(\upsilon_k)$, defined inductively by \[\phi_{k+1} = \upsilon_k \cdot \phi_k \in \mathcal{F}^{k+1} \mathfrak{h}, \quad \lambda_{k + 1} =  \mathrm{BCH}(\lambda_k, - \upsilon_k \otimes 1), \quad \mbox{and} \quad \upsilon_{k+1} = r \left( \lambda_{k + 1} \right)   \ , \] 
	for all $k \geqslant 1$. By induction, we have that $\upsilon_k \in \mathcal{F}^k \mathfrak{h}$, for all $k \geqslant 1$: if $\upsilon_k \in \mathcal{F}^k \mathfrak{h}$, we have \[\upsilon_{k+1} \equiv r \left( \lambda_k, - r(\lambda_k) \otimes 1 \right) \equiv r(\lambda_k) - r(\lambda_k) \equiv 0   \pmod{\mathcal{F}^{k+1} \mathfrak{h}} \ .  \]  In particular, the sequence $(\lambda_k)$ converges to an element $\omega \in \mathfrak{h}$ such that $\omega \cdot \phi = 0$, see \cite[Lemma~6.6]{RNV20}.
\end{proof}

\begin{remark}
	Theorem \ref{cas complet} also follows from \cite[Theorem~1.7]{CPW} since picking a $R$-module retract for the inclusion $R \subset S$ induces a retraction $r : \mathfrak{h} \otimes S \to \mathfrak{h}$ which turns $\mathfrak{h}$ into a retract of $\mathfrak{h} \otimes S$ as a filtered $\mathfrak{h}$-module. 
\end{remark}

\noindent We conclude this section with an example coming from algebraic topology. Using \cite[Theorem 1.5]{berglund15}, we can see fibrations of topological spaces as Maurer-Cartan elements of a certain complete dg Lie algebra. By applying Theorem \ref{cas complet}, we can detect whether a given fibration is trivial up to homotopy.

\begin{definition}[Localizations]
	Let $F$ be a  nilpotent topological space of finite $\mathbb{Q}$-type and let $A(F)$ be its Sullivan model $F$, see \cite[Section~I.9]{FHT01}. Its $\mathbb{Q}$-localization is the space $F_{\mathbb{Q}} = \langle A(F) \rangle$ defined by the geometric realization of the simplicial set \[\mathrm{Hom}_{\mathrm{cdga}} \left(A(F) , \mathcal{A}_{PL}(\Delta^{\bullet}) \right) \ . \] We defined similarly its $\mathbb{R}$-localization $F_{\mathbb{R}}$ as the geometric realization of the simplicial set \[\mathrm{Hom}_{\mathrm{cdga}} \left(A(F) , \mathcal{A}_{PL}(\Delta^{\bullet}) \otimes \mathbb{R} \right) \ . \]
\end{definition}

\begin{theorem}\label{fibrations}
	Let $X$ be a simply connected topological space and let $F$ be a nilpotent space of finite $\mathbb{Q}$-type. A fibration $\xi$ over $X$ with fiber $F_{\mathbb{Q}}$ is trivial up to homotopy if and only if the fibration $\xi \otimes \mathbb{R}$, obtained after extending fiberwise the scalars to $\mathbb{R}$, is trivial up to homotopy.
\end{theorem}

	\begin{proof}
	By Proposition \ref{tech}, it suffices to prove that $\phi  = \varphi - \psi$ is gauge trivial in $\mathfrak{h} = \mathfrak{g}^{\psi}$ if and only if $\phi \otimes 1$ is gauge trivial in $\mathfrak{h}\otimes S$. The direct implication is immediate. Conversely, suppose that there exists $\lambda \in \mathfrak{h}$ such that \begin{equation}\label{eq93}
		\lambda \cdot  \left(\phi \otimes 1\right)  = 0 \quad \mbox{in} \quad \mathfrak{h} \otimes S \ .
	\end{equation}   Let us construct a gauge triviality sequence of $\phi$ as follows. By (\ref{eq93}), we have \[\phi \otimes 1 \equiv d^{\psi \otimes 1} \left(\lambda\right) \pmod{\mathcal{F}^{2} \mathfrak{h} \otimes S}  \ .  \] Let us note that the scalar extension $\mathfrak{h} \to \mathfrak{h} \otimes S,  x \mapsto x \otimes 1$ is injective since $R \to S$ is faithfully flat. Picking a $R$-module retract for the inclusion $R \subset S$ induces a retraction $r : \mathfrak{h} \otimes S \to \mathfrak{h}$. The above equivalence leads to  \[\phi  \equiv d^{\psi } (r(\lambda)) \pmod{\mathcal{F}^{2} \mathfrak{h}} \quad \mbox{and} \quad \vartheta_1 = \left[ \pi_2 (\phi)\right] = 0\ . \] Let us set $\upsilon_1 = r(\lambda)$ and $\phi_2 = \upsilon_1 \cdot \phi$. The element $\lambda_2 \coloneqq \mathrm{BCH}(\lambda, - \upsilon_1 \otimes 1)$ defines a gauge equivalence between $\phi_2 \otimes 1$ and $0$. As before, it leads to \[\phi_2  \equiv d^{\psi} (r(\lambda_2)) \pmod{\mathcal{F}^{2} \mathfrak{h}} \quad \mbox{and} \quad \vartheta_2 = \left[ \pi_3 (\phi_2)\right] =0 \ . \] Let us set $\upsilon_2 = r(\lambda_2)$ and $\phi_3 = \upsilon_2 \cdot \phi_2 .$  The construction can be performed higher up. This leads to a gauge triviality sequence $(\vartheta_k)$ which is an infinite sequence of vanishing homology classes associated with a sequence a Maurer--Cartan element $(\phi_k)$ and gauges $(\upsilon_k)$, defined inductively by \[\phi_{k+1} = \upsilon_k \cdot \phi_k \in \mathcal{F}^{k+1} \mathfrak{h}, \quad \lambda_{k + 1} =  \mathrm{BCH}(\lambda_k, - \upsilon_k \otimes 1), \quad \mbox{and} \quad \upsilon_{k+1} = r \left( \lambda_{k + 1} \right)   \ , \] 
	for all $k \geqslant 1$. By induction, we have that $\upsilon_k \in \mathcal{F}^k \mathfrak{h}$, for all $k \geqslant 1$: if $\upsilon_k \in \mathcal{F}^k \mathfrak{h}$, we have \[\upsilon_{k+1} \equiv r \left( \lambda_k, - r(\lambda_k) \otimes 1 \right) \equiv r(\lambda_k) - r(\lambda_k) \equiv 0   \pmod{\mathcal{F}^{k+1} \mathfrak{h}} \ .  \]  In particular, the sequence $(\lambda_k)$ converges to an element $\omega \in \mathfrak{h}$ such that $\omega \cdot \phi = 0$, see \cite[Lemma~6.6]{RNV20}.
\end{proof}

\begin{remark}
	Theorem \ref{cas complet} also follows from \cite[Theorem~1.7]{CPW} since picking a $R$-module retract for the inclusion $R \subset S$ induces a retraction $r : \mathfrak{h} \otimes S \to \mathfrak{h}$ which turns $\mathfrak{h}$ into a retract of $\mathfrak{h} \otimes S$ as a filtered $\mathfrak{h}$-module. 
\end{remark}

\begin{example}
Let us denote by $\mathrm{Conf}_n(X)$ the space of configurations of $n$ disjoint points in a given space $X$~. For all integers $n,d$, the Fadell--Neuwirth fibration is defined by \[\mathrm{Conf}_{n-1}\left(\mathbb{R}^d\right) \longrightarrow \mathrm{Conf}_n\left(\mathbb{S}^d\right) \longrightarrow \mathbb{S}^d \ ,\] where the second map is the evaluation on the last point of a given configuration, see \cite[Theorem~3]{Fad62}. Let us assume that $d$ is odd and $d \geqslant5$ (the cases $d= 1$ and $d=3$ are trivial using the groups structures of $\mathbb{S}^1$ and $\mathbb{S}^3$). This fibration is trivial after extending fiberwise the scalars to $\mathbb{R}$~, see \cite[Section~4]{Hay22}. Using Theorem \ref{fibrations}, the Fadell--Neuwirth fibration with rationalized fibers is trivial up to homotopy.  
\end{example}

\section{\textcolor{bordeau}{Obstruction sequences to homotopy equivalences}}\label{section2}

Let \( A \) be a chain complex over a commutative ring \( R \), and let \( \varphi \) and \( \psi \) be two algebraic structures defined on \( A \) (e.g., dg associative algebras, dg Frobenius algebras, etc.). As recalled in Section \ref{2.1}, such algebraic structures can be viewed as Maurer–Cartan elements of a certain complete differential graded Lie algebra \( \mathfrak{g}_A \). When \( R \) is a \( \mathbb{Q} \)-algebra, the existence of a homotopy equivalence between \( \varphi \) and \( \psi \) can be reformulated as a gauge equivalence problem (see Section \ref{2.2}). In Section \ref{2.35}, we apply the obstruction theory developed in Section \ref{section1} to study the existence of homotopy equivalences in characteristic zero. The Lie algebra \( \mathfrak{g}_A \) is, in fact, a weight-graded convolution dg Lie-admissible algebra. This enriched structure allows us to generalize the results of Section \ref{section1} in order to study formality over any commutative ground ring \( R \), as discussed in Sections \ref{2.4} and \ref{2.5}.

\begin{remark}
	This section generalizes to any algebraic structure the constructions of \cite{HJ79} which addresses the case of commutative algebras in characteristic zero. It generalizes to any commutative ground ring and algebraic structure the constructions of \cite{Sal17} focusing on the formality of associative, commutative and Lie algebras in characteristic zero. 
\end{remark}

\noindent Let us fix $R$ a commutative ring and let $A$ be a differential graded $R$-module. Let $\C$ be a reduced conilpotent dg coproperad over $R$. Let us denote the coaugmentation coideal $ \overline{\C}$~. Let us fix a filtration of this cooperad \[ \mathcal{F}^1 \C \subseteq \mathcal{F}^2 \C \subseteq \cdots \subseteq  \mathcal{F}^n \C \subseteq \cdots  \] that is exhaustive: $\mathrm{colim}_n \mathcal{F}^n \C = \C$. For instance, one can take $\mathcal{F}$ to be the coradical filtration.

\begin{example}
	Among possible choices for such data, we have: 
	\begin{enumerate}
		\item The bar construction $\C \coloneqq \Bar \P$ of a reduced properad $\P$ in $\mathrm{gr Mod}_R$~, see \cite{PHC}~.
		\item The Koszul dual coproperad $\C := \P^{\antishriek}$, if $\P$ is a homogeneous Koszul properad, see \cite{Vallette_2007}, or an inhomogeneous quadratic Koszul properad, see \cite[Appendix~A]{GCTV12}.		
		\item The quasi-planar cooperad $\C = \Bar \left(\mathcal{E} \otimes \P\right) $ where $\mathcal{E}$ denotes the Barratt-Eccles operad and $\P$ is any reduced operad, see \cite{LRL23} \ . 
	\end{enumerate}
\end{example}

\subsection{Deformation complex of algebras over properads}\label{2.1}

In this section, we recall the definition of an $\Cobar \C$-algebra structure on $A$, as Maurer--Cartan element of the associated convolution algebra $\mathfrak{g}_A$~. We refer the reader to \cite{PHC} for more details.

\begin{definition}[dg Lie-admissible algebra]
	A \emph{complete dg Lie-admissible algebra} $(\mathfrak{g},  \star, d , \mathcal{F})$ is a graded $R$-module $\mathfrak{g}$ equipped with 
\begin{itemize}
	\item[$\centerdot$]  a linear map of degree zero $\star : A\otimes A \to A$ whose skew-symmetrization \[[x,y] := x \star y - (-1)^{|x||y|} y \star x \] satisfies the Jacobi identity, i.e. induces a dg Lie algebra structure. 
	\item[$\centerdot$] a differential $d : \mathfrak{g} \to \mathfrak{g}$, i.e. a degree $-1$ derivation that squares to zero.
	\item[$\centerdot$] a complete descending filtration $\mathcal{F}$, i.e. a decreasing chain of sub-complexes \[\mathfrak{g} = \mathcal{F}^1 \mathfrak{g} \supset \mathcal{F}^2 \mathfrak{g} \supset \mathcal{F}^3 \mathfrak{g} \supset \cdots  \] compatible with the product $\mathcal{F}^n \mathfrak{g} \star \mathcal{F}^m \mathfrak{g} \subset \mathcal{F}^{n+m} \mathfrak{g}$, and such that the canonical projections $\pi_n : \mathfrak{g} \twoheadrightarrow  \mathfrak{g} /\mathcal{F}^n \mathfrak{g}$ lead to an isomorphism \[\pi : \mathfrak{g} \to \lim_{n \in \mathbb{N}} \mathfrak{g} / \mathcal{F}^n\mathfrak{g} \ . \] 
\end{itemize} Its set of \emph{Maurer--Cartan elements} is defined as $\mathrm{MC(\mathfrak{g})} \coloneqq \lbrace x \in \mathfrak{g}_{-1} \mid dx + x\star x = 0 \rbrace \ . $
\end{definition}

	\begin{definition}[Infinitesimal composition]
The \emph{infinitesimal composition} of two $\mathbb{S}$-bimodules $\mathcal{M}$ and $\mathcal{N}$ is the sub-$\mathbb{S}$-bimodule composed of the parts which are linear in $\mathcal{M}$ and in $\mathcal{N}$: \[\mathcal{M} \underset{(1,1)}{\boxtimes} \mathcal{N} \subseteq (\mathcal{I} \otimes \mathcal{M}) \boxtimes (\mathcal{I} \otimes \mathcal{N})\ .\] Let $(\mathbf{\P}, \gamma, \eta)$ be a properad and let $(\C, \Delta, \epsilon, \eta )$ be a coaugmented cooperad.  
\begin{itemize}
	\item[$\centerdot$] The \emph{infinitesimal composition map} of $\P$ \[\gamma_{(1,1)} : \P \underset{(1,1)}{\boxtimes} \P \longrightarrow (\mathcal{I} \otimes \P) \boxtimes (\mathcal{I} \otimes \P) \xrightarrow{(\eta + \mathrm{id}) \boxtimes (\eta + \mathrm{id})}  \P \boxtimes \P \overset{\gamma}{\longrightarrow} \P \ .  \]
	\item[$\centerdot$]  The \emph{infinitesimal decomposition map} $\Delta_{(1,1)}$ of $\C$   is given by
	\[ \mathcal{I} \overset{\cong}{\longrightarrow} \mathcal{I} \boxtimes \mathcal{I} \quad \mbox{and} \quad \overline{\C} \overset{\Delta}{\longrightarrow} \C \boxtimes \C \twoheadrightarrow \overline{\C} \underset{(1,1)}{\boxtimes} \overline{\C} \ . \] \end{itemize} 
\end{definition}

\begin{definition}[Convolution dg Lie-admissible algebra]
	The \emph{convolution dg Lie-admissible algebra} associated to $\C$ and $A$ is the complete dg Lie-admissible algebra \[\mathfrak{g}_{A} := \left( \Hom_{\mathbb{S}} \left(\overline{\C}, \mathrm{End}_A\right), \star, d , \mathcal{F} \right) \ , \] 	whose the underlying space is \[\Hom_{\mathbb{S}}\left(\overline{\C}, \End_A \right) \coloneqq \prod_{m,n \geqslant0} \Hom_{\mathbb{S}_m^{op} \times \mathbb{S}_n}\left(\overline{\C}(m,n),\End_A(m,n) \right)\ ,\] equipped with the Lie-admissible product \[\varphi \star \psi \coloneqq \overline{\C} \xrightarrow{\Delta_{(1,1)}} \overline{\C} \underset{(1,1)}{\boxtimes} \overline{\C} \xrightarrow{\varphi \underset{(1,1)}{\boxtimes}  \psi} \End_A \underset{(1,1)}{\boxtimes} \End_A \xrightarrow{\gamma_{(1,1)}} \End_A \ , \] the differential \[d (\varphi) \coloneqq d_{\End_A} \circ \varphi - (-1)^{|\varphi|} \varphi \circ d_{\overline{\C}} \] and the complete descending filtration given by \begin{equation}\label{filtration}
	\mathcal{F}^n \mathfrak{g}_{A} \coloneqq \prod_{ k \geqslant n  } \Hom_{\mathbb{S}_m^{op} \times \mathbb{S}_n}\left(\mathcal{F}^k \overline{\C},\End_A \right) \ .
	\end{equation} It embeds into the complete dg Lie-admissible algebra made up of all the maps from $\C$~, i.e. \[\mathfrak{g}_{A} \hookrightarrow \mathfrak{a}_{A} \coloneqq \left( \Hom_{\mathbb{S}} \left(\C, \mathrm{End}_A\right), \star, d \right) \ . \]  More generally, let us consider the complete dg  Lie-admissible algebra \[	\mathfrak{a}_B^A \coloneqq \Hom_{\mathbb{S}} \left(\C, \mathrm{End}_B^A\right) \ .\]  We refer the reader to \cite[Proposition 11]{Merkulov_2009} for more details.
\end{definition}

\begin{proposition}
	A \emph{$\Cobar \C$-algebra structure} $\varphi$ on $A$ is a Maurer--Cartan element in $\mathfrak{g}_A$~, i.e. \[\varphi \in \mathrm{MC}(\mathfrak{g}_{A}) \coloneqq \{\varphi \in \mathfrak{g}_{A} \mid |\varphi|=-1, \; d(\varphi) + \varphi \star \varphi = 0\} \ .\] 
\end{proposition}

\begin{proof}
	See \cite[Proposition~3.6]{PHC}. 
\end{proof}

\begin{definition}[Left/right actions]
For all $f : \C \to \End_B^A$,  the \emph{left action} of $\psi \in  \mathfrak{a}_B$ on $f$ and the \emph{right action} of $\varphi \in  \mathfrak{a}_A$ on $f$ are defined by 	\[\begin{array}{ccccccccc}
	\psi 	\lhd f & \coloneqq & \overline{\C} & \xrightarrow{\Delta_{(1,*)}} & \overline{\C} \underset{(1,*)}{\boxtimes} \C &  \xrightarrow{\psi \underset{(1,*)}{\boxtimes} f} &  \End_B \underset{(1,*)}{\boxtimes} \End_B^A & \longrightarrow & \End_B^A \ ,  \\
	& & \mathcal{I} & \overset {  \cong }{ \ \longrightarrow \ } & \mathcal{I} \ \boxtimes \ \mathcal{I} &  \xrightarrow{\varphi \; \boxtimes \; f} & \End_B \ \boxtimes \ \End_B^A & \longrightarrow & \End_B^A \ , \\ 
	& & & & & & & & \\ f 	\rhd \varphi & \coloneqq & \overline{\C} & \xrightarrow{\Delta_{(*,1)}} & \C \underset{(*,1)}{\boxtimes} \overline{\C}  &  \xrightarrow{f \underset{(*,1)}{\boxtimes} \varphi} &  \End_B^A \underset{(*,1)}{\boxtimes} \End_A & \longrightarrow & \End_B^A \ ,  \\
	& & \mathcal{I} & \overset{\cong}{\ \longrightarrow \ } & \mathcal{I} \ \boxtimes \ \mathcal{I} &  \xrightarrow{f \; \boxtimes \; \varphi} & \End_B^A \ \boxtimes \End_A & \longrightarrow & \End_B^A \ .
\end{array}\] 
\end{definition}

\begin{remark}
	The operator $\rhd$ defines a left $L_{\infty}$-module structure of the dg Lie algebra $\mathfrak{a}_B$ on $\mathfrak{a}_B^A$ and the operator $\lhd$ defines a right $L_{\infty}$-module structure of  $\mathfrak{a}_A$ on $\mathfrak{a}_B^A$.
\end{remark}

\begin{definition}[$\infty$-morphism]
	Let $(A,\varphi)$ and $(B,\psi) $ be two $\Cobar \C$-algebra structures. An \emph{$\infty$-morphism} $\varphi \rightsquigarrow \psi$ is a degree zero morphism $f \in \mathfrak{a}^A_B$ such that \[f 	\rhd \varphi -  \psi	\lhd f = d(f) \  .\] We consider the first component \[f^{(0)} : \mathcal{I} \hooklongrightarrow \C \overset{f}{\longrightarrow} \End^A_B \ . \] The $\infty$-morphism $f : \varphi \rightsquigarrow \psi$ is
	\begin{itemize}
		\item[$\centerdot$] an \emph{$\infty$-quasi-isomorphism} if $f^{(0)} \in \Hom(A,B)$ is a quasi-isomorphism; 
		\item[$\centerdot$] an \emph{$\infty$-isomorphism} if $f^{(0)} \in \Hom(A,B)$ is an isomorphism; 
		\item[$\centerdot$] an \emph{$\infty$-isotopy} if $A = B$ and $f^{(0)} = \mathrm{id}_A$~. 
	\end{itemize}
 The set of all the $\infty$-isotopies is denoted by $\Gamma_A$~.
\end{definition}

\begin{theorem}[{\cite[Proposition~3.19 and Theorem~3.22]{PHC}}]\label{invertibility prop} \leavevmode

	\begin{enumerate}
		\item The composite of $f : \C \to \mathrm{End}_B^A$ and $g : \C \to \mathrm{End}_C^B$ is defined by  \[g \circledcirc f \coloneqq \C \xrightarrow{\ \Delta \ } \C \ \boxtimes \ \C \xrightarrow{g \ \boxtimes \  f} \End_C^B \ \boxtimes \ \End_B^A \xrightarrow{\ \gamma \ } \End_C^A \ . \]
		\item  The composite of $f : \varphi \rightsquigarrow \psi$ and $g : \psi \rightsquigarrow \phi$ is an $\infty$-morphism $g \circledcirc f :  \varphi  \rightsquigarrow \phi$.
		\item If $f^{(0)}$ is an isomorphism, $f$ admits a unique inverse with respect to $\circledcirc$~, denoted $f^{-1}$~. 
		\item The set of $\infty$-isotopies $\Gamma_A$ forms a group with respect to $\circledcirc$. 
	\end{enumerate}	
\end{theorem}

\noindent By \cite{graphexp}, the group of $\infty$-isotopies $\Gamma_A$ is isomorphic to the gauge group of the dg Lie admissible algebra $\mathfrak{g}_A$.

\begin{theorem}[{\cite[Theorem~2.24]{graphexp}}]\label{graphexp2}
If $R$ is a $\mathbb{Q}$-algebra, the gauge group associated to $\mathfrak{g}_A$ and the group $\Gamma_A$ are isomorphic though the \emph{graph exponential/logarithm maps}, \[\mathrm{exp} : ((\mathfrak{g}_A)_0, \mathrm{BCH}, 0) \cong (\Gamma_A, \circledcirc, 1) : \mathrm{log}\ . \] 
\end{theorem}

\subsection{The adjoint action of $\infty$-isomorphisms}\label{2.3}

In characteristic zero, the group of $\infty$-isotopies $\Gamma_A$ acts on the set of $\Omega \C$-algebra structures on $A$ via the gauge action and the isomorphism established in Theorem \ref{graphexp2}. The goal of this section is to generalize this result to an arbitrary coefficient ring by making explicit the action of $\infty$-isomorphisms on $\Omega \C$-algebra structures in the properadic case.

\begin{notation}
	Let $(A,\varphi)$ and $(B,\psi)$ be two $\Omega \C $-algebra structures. Let us fix the  elements
	\[  f, \ell  : \C \to \mathrm{End}_B^A ,\; g , h : \C \to \mathrm{End}_C^B,  \; k : \C \to \mathrm{End}_A^E, \; \mbox{and} \; x,y : \C \to \mathrm{End}_D^C   \ .   \]  Let us recall the notation $g \circledcirc (f; \ell)$ of \cite{PHC} for $g$ applied everywhere at the bottom level and $f$ applied everywhere at the top level except at one place where $\ell$ is applied. The operation $(g ; h)   \circledcirc f$ is defined similarly. By definition, the following equalities hold \[\psi \lhd f = \left(1 ; \psi  \right)\circledcirc f, \quad  f \rhd \varphi = f \circledcirc \left(1 ; \varphi \right), \quad  x \star y = (1; x) \circledcirc \left(1 ; y \right) \ . \] 
\end{notation}

\begin{definition}[Adjoint action]
Let $f : \C \to \mathrm{End}_B^A $ of degree zero be such that $f^{(0)}$ is an isomorphism. For all $\varphi \in \mathfrak{g}_A$, let us denote  
\[ f \cdot \varphi \coloneqq \mathrm{Ad}_f \left(  \varphi  \right) -  \left(f ;d(f) \right) \circledcirc f^{-1} \ , \] where $\mathrm{Ad}_f \left( \varphi  \right) \coloneqq (f ; f \rhd  \varphi) \circledcirc f^{-1}$.  
\end{definition}

\noindent The goal of this section is to establish that $f \cdot x$ defines an action of $\infty$-isomorphisms on $\Omega \C$-structures. To this aim, we start by establishing the compatibility relations satisfied by the operations at the level of the space $\Hom_{\mathbb{S}} \left(\C, \mathrm{End}_B^A\right)$.

\begin{lemma}\label{cal}  The following identities hold
	\begin{enumerate}
		\item $(x;y) \circledcirc (g \circledcirc f) = (x \circledcirc g ; (x;y) \circledcirc g) \circledcirc f$ 
		\item $ (x \circledcirc g ; (x; y) \circledcirc g) \circledcirc (f ; \ell) = (x; y)  \circledcirc (g \circledcirc f ; g \circledcirc (f; \ell))$
		\item $ (g  \circledcirc f ; g  \circledcirc (f ; \ell) ) \circledcirc k = g \circledcirc (f \circledcirc k ;  (f; \ell) \circledcirc k)$
		\item $ \psi \lhd \left(f \circledcirc k \right) = \left(f ; \psi \lhd f \right) \circledcirc k $
		\item $  \left(g \circledcirc f \right) \rhd \varphi = g \circledcirc \left( f ; f \rhd \varphi \right) $ 
	\end{enumerate}
	\noindent If furthermore, the elements $f$ and $g$ are invertible with respect to $\circledcirc$, then
	\begin{enumerate}
		\item[$(6)$] $\left((f ; \ell) \circledcirc f^{-1} \right) \lhd f =  \ell    $     
		\item[$(7)$] $ \mathrm{Ad}_g \left(   (f;\ell) \circledcirc f^{-1} \right)  = \left(g \circledcirc f ; g \circledcirc (f; \ell)\right) \circledcirc (f^{-1} \circledcirc g^{-1} )   $
		\item[$(8)$] $ (f; \ell) \circledcirc (f^{-1}; \ell') = \left((f; \ell) \circledcirc f^{-1} \right) \star \left( f \circledcirc (f^{-1}; \ell')  \right)$  
		\item[$(9)$] $\mathrm{Ad}_f \left( \varphi  \right) = f \circledcirc \left(f^{-1}; \varphi \lhd f^{-1}\right) $.
	\end{enumerate}
\end{lemma}

\begin{proof} \leavevmode
	\begin{enumerate}
		\item The formula follows from the coassociativity of the decomposition map $\Delta$ and the following isomorphism \[\left( \C ; \C  \right) \boxtimes \left( \C \boxtimes \C \right) \cong \left( \C \boxtimes \C ; \left( \C ; \C  \right) \boxtimes \C \right) \boxtimes \C \ . \] 
		\item The formula follows from the coassociativity of the decomposition map $\Delta$ and the following isomorphism \[\left(\C   \boxtimes \C ; \left( \C ; \C  \right) \boxtimes \C \right) \boxtimes \left( \C ; \C \right) \cong \left( \C ; \C  \right) \boxtimes \left( \C \boxtimes \C ;\C \boxtimes \left( \C ; \C  \right)  \right)  \ . \] 
		\item The formula follows from the coassociativity of the decomposition map $\Delta$ and the following isomorphism \[\left( \C \boxtimes \C ; \C \boxtimes \left(\C ; \C \right)  \right) \boxtimes \C \cong \C \boxtimes \left( \C \boxtimes \C ; \left( \C ; \C  \right) \boxtimes \C \right)  \ . \] 
		\item This follows from (1) by considering $x = 1$ and $y = \psi$. 
		\item This follows from (2) by considering $x =y =1$, $g=f$, $f=1$ and $\ell = \varphi$. 

		\item By formula (1) with $x = f$, $y = \ell$ and $g = f^{-1}$, we have  \[\left((f ; \ell) \circledcirc f^{-1} \right) \lhd f = \left(1; (f ; \ell) \circledcirc f^{-1}  \right) \circledcirc f \overset{(1)}{=} (f ; \ell) \circledcirc 1 \ . \] Since all operations in a coproperad are connected, we have $ (f ; \ell) \circledcirc 1 = \ell$.  
\item By formula (1) with $x = g \circledcirc f$, $y = g \circledcirc (f ; \ell )$, $g =  f^{-1}$ and $f = g^{-1}$, we have \[\left(g \circledcirc f ; g \circledcirc (f; \ell)\right) \circledcirc \left(f^{-1} \circledcirc g^{-1} \right) \overset{(7)}{=} \left(g ; \left(g \circledcirc f ; g \circledcirc (f ; \ell ) \right) \circledcirc  f^{-1} \right) \circledcirc  g^{-1} \ . \] The result follows from formula (3) with $k= f^{-1}$ since \[\left(g \circledcirc f ; g \circledcirc (f ; \ell ) \right) \circledcirc  f^{-1} \overset{(3)}= g \circledcirc \left(1 ; (f;\ell) \circledcirc f^{-1} \right) = g \rhd \left( (f;\ell) \circledcirc f^{-1} \right) \ . \]
\item By formula (2) with $(x;y) = \left(1 ; (f; \ell) \circledcirc f^{-1} \right)$, $g = f$, $f = f^{-1}$ and $\ell = \ell'$, we have 
\[ \left(1 ; (f; \ell) \circledcirc f^{-1} \right) \circledcirc \left(1 ; f \circledcirc (f^{-1}; \ell')  \right) \overset{(2)}{=} \left(f ; \left( (f; \ell) \circledcirc f^{-1} \right) \lhd f \right) \circledcirc (f^{-1}, \ell' )  \ .\] Formula (6) yields the desired formula. 
\item This follows from (3) with $g = f$, $f = 1$, $\ell = \varphi$ and $k = f^{-1}$.  \qedhere
	\end{enumerate}
\end{proof}

\begin{lemma}\label{cal2}
	If $f$ and $g$ are of degree $0$, the following identities hold:
	\begin{enumerate}
		\item $d (g \circledcirc f) = (g ; d(g)) \circledcirc f +  g \circledcirc (f; d(f)) $;
		\item $ (f ; d(f)) \circledcirc f^{-1} = - f \circledcirc \left(f^{-1}, d \left(f^{-1}\right)\right)$.
	\end{enumerate}  
\end{lemma}

\begin{proof}
	The first point follows from the fact that $d$ is a coderivation. The second point follows from applying (1) to $f^{-1} \circledcirc f = 1$. 	
\end{proof}

\begin{proposition}\label{Adf properadic0}
 Let $f : \C \to \mathrm{End}_B^A $ of degree zero be such that $f^{(0)}$ is an isomorphism. 
	\begin{enumerate}
		\item  The element $f$ defines an $\infty$-isomorphism $\varphi \rightsquigarrow \psi$ if and only if we have $ f \cdot \varphi = \psi$~.
		\item For all $g : \C \to \End_C^B $ of degree zero such that $g^{(0)}$ is an isomorphism, we have \[\mathrm{Ad}_{g \circledcirc f} =  \mathrm{Ad}_g \circ \mathrm{Ad}_f \quad \mbox{and} \quad  g \cdot (f \cdot \varphi) = (g \circledcirc f) \cdot \varphi \ .\]
		\item The following equality holds \begin{equation*}\label{truand}
			d \left( \mathrm{Ad}_{f} (\varphi) \right) =   	\mathrm{Ad}_{f} (d  (\varphi)) + \left[ \left(f ;d(f) \right) \circledcirc f^{-1} , \mathrm{Ad}_{f} (\varphi) \right]  \ .
		\end{equation*}
		\item The element $f \cdot \varphi$ defines an $\Omega \C $-algebra structure on $B$. 
		\item The group $\Gamma_A$ acts on the set of $\Omega \C $-algebra structures on $A$ under $f \cdot \varphi$. 
	\end{enumerate}
\end{proposition}

\begin{proof} 		
	Let us prove the point (1). If $f$ defines an $\infty$-isomorphism $\varphi \rightsquigarrow \psi$, we have \begin{equation}\label{infini}
		f 	\rhd \varphi -  \psi	\lhd f = d(f) \  .
	\end{equation}  Applying the operator $\left(f, -\right) \circledcirc f^{-1}$ on both sides leads to \[f 	\cdot  \varphi = \left(f, \psi	\lhd f \right) \circledcirc f^{-1}   \  .\] Using formula $(4)$ of Lemma \ref{cal} with $k  = f^{-1}$, we have $f 	\cdot  \varphi = \psi$. Conversely, suppose that we have $f \cdot \varphi = \psi$. Applying the operator $\lhd f$ on both sides leads to equality (\ref{infini}), by formula (6) of Lemma \ref{cal}.  Let us prove the point (2). On the one hand, we have $\left(g \circledcirc f \right) \rhd \varphi = g \circledcirc \left( f ; f \rhd \varphi \right) $ by formula (5)  of Lemma \ref{cal}. Applying formula (7) of Lemma \ref{cal}, with  $\ell = f \rhd \varphi$, we obtain
		\[ \mathrm{Ad}_{g \circledcirc f}(\varphi) = \left(g \circledcirc f ; (g \circledcirc f)	\rhd \varphi  \right) \circledcirc \left( f^{-1} \circledcirc g^{-1} \right) = \mathrm{Ad}_g \circ \mathrm{Ad}_f (\varphi) \ . \] On the other hand, we have \[ (g \circledcirc f) \cdot \varphi =  \mathrm{Ad}_g \circ \mathrm{Ad}_f (\varphi)   - \left(g \circledcirc f ;  d(g \circledcirc f) \right) \circledcirc f^{-1} \circledcirc g^{-1} \ . \]  Applying formula (1) of Lemma \ref{cal2} leads to the equality \[ d(g \circledcirc f) =  (g ; d(g)) \circledcirc f +  g \circledcirc (f,d(f)) \ . \] By formula (1) of Lemma \ref{cal}, with $y = d(g)$, $x = g$, $g = f$, $f = f^{-1} \circledcirc g^{-1}$, we have	\[\left(g \circledcirc f ;  (g ; d(g)) \circledcirc f \right) \circledcirc f^{-1} \circledcirc g^{-1}  = \left(g ;d(g) \right) \circledcirc g^{-1} \ .\] Applying formula (7) of Lemma \ref{cal}, with $\ell = d(f)$, we have \[\left(g \circledcirc f ;   g \circledcirc (f,d(f)) \right) \circledcirc f^{-1} \circledcirc g^{-1}  =   \mathrm{Ad}_g  \left(\left(f ;d(f) \right) \circledcirc f^{-1}\right) \ .\] This gives the desired result since we have  \[
		g \cdot (f \cdot \varphi)  = \mathrm{Ad}_g \circ \mathrm{Ad}_f (\varphi)  -  \mathrm{Ad}_g  \left(\left(f ;d(f) \right) \circledcirc f^{-1}\right)  -  \left(g ;d(g) \right) \circledcirc g^{-1} \ . \]  Let us prove the point (3). Since $d$ is a coderivation, the following equality holds \[d \left( \mathrm{Ad}_f \left( \varphi  \right) \right) =   \mathrm{Ad}_f \left( d \left( \varphi  \right) \right) + (-1)^{\varphi}\left(f ; f \rhd \varphi \right) \circledcirc \left(f^{-1}; d(f^{-1})\right) + \left(f; d(f)  \right) \circledcirc \left(f^{-1}; \varphi \lhd f^{-1} \right) \ . \]   By formula (8) of Lemma \ref{cal} with $\ell = d(f)$ and $\ell' = \varphi \lhd f^{-1}$ and formula (9), we have \[\left(f; d(f)  \right) \circledcirc \left(f^{-1}; \varphi \lhd f^{-1} \right) = \left(\left(f ;d(f) \right) \circledcirc f^{-1}\right) \star \mathrm{Ad}_f \left(\varphi \right) \ .  \] Similarly, by formula (8) of Lemma \ref{cal} with $\ell = f \rhd \varphi$ and $\ell' = d(f^{-1})$ and formula (2) of Lemma \ref{cal2}, we have \[\left(f ; f \rhd \varphi \right) \circledcirc \left(f^{-1}; d(f^{-1})\right) = - \mathrm{Ad}_f \left(\varphi \right) \star \left(\left(f ;d(f) \right) \circledcirc f^{-1}\right) \ .  \] Let us prove the point (4), i.e. that $f \cdot \varphi$ satisfies the Maurer--Cartan equation that is \begin{equation}\label{MCeq}
			d(f \cdot \varphi) + (f \cdot \varphi) \star (f \cdot \varphi) = 0 \ .
		\end{equation}  By Point (3) of Proposition \ref{Adf properadic0}, the following equality holds \[d(f \cdot \varphi) = \mathrm{Ad}_{f} (d  (\varphi)) + \left[ \left(f ;d(f) \right) \circledcirc f^{-1} , \mathrm{Ad}_{f} (\varphi) \right]  - d \left(\left(f ;d(f) \right) \circledcirc f^{-1}  \right) \ . \] Since the differential squares to zero, we have  \[d \left(\left(f ;d(f) \right) \circledcirc f^{-1}  \right) = - \left(f ;d(f) \right) \circledcirc \left(  f^{-1}, d\left(f^{-1} \right)  \right) \ . \] By formula $(8)$ of Lemma \ref{cal} with $\ell = d(f)$ and $\ell' = d \left(f^{-1} \right)$ and formula (2) of Lemma \ref{cal2}, we have \[d \left(\left(f ;d(f) \right) \circledcirc f^{-1}  \right)  = \left( (f; d(f) ) \circledcirc f^{-1} \right) \star \left( (f; d(f)) \circledcirc f^{-1}\right) \ . \] On the other side, we have \[(f \cdot \varphi) \star (f \cdot \varphi) = \mathrm{Ad}_{f}(\varphi) \star \mathrm{Ad}_{f}(\varphi) - \left[ \left(f ;d(f) \right) \circledcirc f^{-1} , \mathrm{Ad}_{f} (\varphi) \right] + d \left(\left(f ;d(f) \right) \circledcirc f^{-1}  \right) \ .  \] As observed in the proof of \cite[Theorem 5.1]{PHC}, we have $ f \rhd (\varphi \star \varphi) = (f ; f \rhd \varphi ) \rhd \varphi $. Symmetrically, we also have \[\left( \mathrm{Ad}_{f}(\varphi) \star \mathrm{Ad}_{f}(\varphi) \right) \lhd f =  \mathrm{Ad}_{f}(\varphi) \lhd \left(f;  \mathrm{Ad}_{f}(\varphi) \lhd f \right) \overset{(6)}{=} \mathrm{Ad}_{f}(\varphi) \lhd \left(f;   f \rhd \varphi  \right)   \   \] which we simplify by using formula (6) of Lemma \ref{cal}. As also observed in the proof of \cite[Theorem 5.1]{PHC}, we have \[\mathrm{Ad}_{f}(\varphi) \lhd \left(f;   f \rhd \varphi  \right) = \left(f ; \mathrm{Ad}_{f}(\varphi) \lhd  f \right) \rhd \varphi \overset{(6)}{=} \left(f ; f \rhd \varphi \right) \rhd \varphi \ . \] Applying the operator $(f; -) \circledcirc f^{-1}$ on both sides and using formula (4) of Lemma \ref{cal} gives \[ \mathrm{Ad}_{f}(\varphi) \star \mathrm{Ad}_{f}(\varphi)   = \left( f; \left(f ; f \rhd \varphi \right) \rhd \varphi \right) \circledcirc f^{-1} = \mathrm{Ad}_{f}(\varphi \star \varphi)   \ .   \]  Equation (\ref{MCeq}) now follows from the Maurer--Cartan equation of $\varphi$. \noindent Let us now prove point (5). Applying formula (2) of Lemma \ref{cal2} with $f=g= 1$ leads to $d(1) = 0$. We get $1 \cdot \varphi = \varphi$. Together with Point (2) of Proposition \ref{Adf properadic0}, this proves that the group $\Gamma_A$ acts on the set of Maurer-Cartan elements $\mathrm{MC}(\mathfrak{g}_A)$ under $f \cdot \varphi$. 
\end{proof}

\begin{proposition}\label{Adf properadic}
	Let  $f : \varphi \rightsquigarrow \psi$ be an $\infty$-isomorphism. It induces an isomorphism  \[\mathrm{Ad}_f : \mathfrak{g}_A^{\varphi} \to \mathfrak{g}_B^{\psi}, \quad x \mapsto (f ; f \rhd  x) \circledcirc f^{-1} \ ,\] of dg Lie algebras  whose inverse is given by $\mathrm{Ad}_{f^{-1}}$~. 
\end{proposition}

\begin{proof}	
	Let us prove that $\mathrm{Ad}_f$ is a morphism of Lie algebras. For all $x,y$ in $\mathfrak{g}_A$, we claim that  \begin{equation}\label{dua}
	  (f ; f \rhd x) \rhd y  - (-1)^{|x||y|} 	  (f ; f \rhd y) \rhd x = f \rhd (x \star y) -   (-1)^{|x||y|} f \rhd (y \star x) \ .
	\end{equation} In the component $(f ; f \rhd x) \rhd y $, one is applying twice the comonadic decomposition map in order to produce all the graphs with two top vertices
	labeled by $x$ and $y$ and a bottom level of vertices labeled by $f$. When the two top vertices do not sit one
	above the other and can be vertically switched, the corresponding terms cancel with the ones in $(f ; f \rhd y) \rhd x$. Only remain the terms where the two vertices do sit one above the other, which is obtained by \[ f \rhd (x \star y) -   (-1)^{|x||y|} f \rhd (y \star x) \ .\] By Equation (\ref{dua}), we have that \[\mathrm{Ad}_f \left( \left[x,y\right] \right) = (f ; (f ; f \rhd x) \rhd y ) \circledcirc f^{-1} - (-1)^{|x||y|} 	 (f ; (f ; f \rhd y) \rhd x ) \circledcirc f^{-1} \ . \] Considering the vertical symmetry of Equation (\ref{dua}) applied to $\mathrm{Ad}_f (x )$ and $ \mathrm{Ad}_f (y )$, the following equation holds as well \[\mathrm{Ad}_f (x ) \lhd	(f ; \mathrm{Ad}_f (y ) \lhd f)   - (-1)^{|x||y|} 	 \mathrm{Ad}_f (y ) \lhd	(f ; \mathrm{Ad}_f (x ) \lhd f)  = \left[\mathrm{Ad}_f (x ), \mathrm{Ad}_f (y ) \right] \lhd f  \ .   \] As observed in the proof of \cite[Theorem 5.1]{PHC}, we have \[\mathrm{Ad}_{f}(x) \lhd \left(f;   f \rhd y \right) = \left(f ; \mathrm{Ad}_{f}(x) \lhd  f \right) \rhd y \overset{(6)}{=} \left(f ; f \rhd x \right) \rhd y \ , \] which we simplify by using formula (6) of Lemma \ref{cal}. Similarly, we get $\mathrm{Ad}_{f}(y) \lhd \left(f;   f \rhd x \right) = \left(f ; f \rhd y \right) \rhd x$ and thus \[\left[\mathrm{Ad}_f (x ), \mathrm{Ad}_f (y ) \right] \lhd f  = (f ; f \rhd x) \rhd y  - (-1)^{|x||y|} 	  (f ; f \rhd y) \rhd x \ . \]  Applying the operator $(f; -) \circledcirc f^{-1}$ on both sides and using formula (4) of Lemma \ref{cal} implies that $\mathrm{Ad}_f \left( \left[x,y\right] \right) = \left[\mathrm{Ad}_f (x ), \mathrm{Ad}_f (y ) \right]  $.   Let us now prove that $\mathrm{Ad}_f$ is a morphism of chain complexes that is \[ \mathrm{Ad}_{f} \left(d^{\varphi} (x) \right)  = d^{\psi} \left( \mathrm{Ad}_{f} (x) \right) \ .   \] Since $\psi =  \mathrm{Ad}_{f} (\varphi) - \left(f ;d(f) \right) \circledcirc f^{-1} $ by point (1) of Proposition \ref{Adf properadic0}, the result follows from point (3) of Proposition \ref{Adf properadic0}. 
	Finally, by point (2) of Proposition \ref{Adf properadic0}, we have  \[\mathrm{Ad}_f \circ \mathrm{Ad}_{f^{-1}} = \mathrm{Ad}_{1} = \mathrm{id}\ ,\] the operator $\mathrm{Ad}_f$ is an isomorphism of dg Lie algebras whose inverse is given by $\mathrm{Ad}_{f^{-1}} $. 
\end{proof}

\subsection{Gauge homotopy equivalences}\label{2.2}

In this section, we study the notion of homotopy equivalence between two algebraic structures defined on a fixed chain complex \( A \). More precisely, we introduce the closely related concept of \emph{gauge homotopy equivalence}.

\begin{definition}[Homotopy equivalences]
Two $\Omega \C$-algebra structures $(A, \varphi)$ and $(B, \psi)$ are 
\begin{itemize}
	\item[$\centerdot$] \emph{homotopy equivalent} if there exists a zig-zag of quasi-isomorphisms of $\Cobar \C$-algebras \[(A, \varphi) \; \overset{\sim}{\longleftarrow} \;  \cdot \;  \overset{\sim}{\longrightarrow} \;  \cdots \;  \overset{\sim}{\longleftarrow} \; \cdot \;  \overset{\sim}{\longrightarrow} \; (B, \psi) \ .\]
	\item[$\centerdot$] \emph{gauge homotopy equivalent} if there exists an $\infty$-quasi-isomorphism of $\Cobar \C$-algebras 
		\[\begin{tikzcd}[column sep=normal]
			(A,\varphi) \ar[r,squiggly,"\sim"] 
			& (B, \psi) \ .
				\end{tikzcd} 
		\]
		\item[$\centerdot$] \emph{gauge $n$-homotopy equivalent} if $(A,\varphi)$ is gauge homotopy
		equivalent to an $\Omega \C$-algebra $(B,\phi)$ such that $\phi - \psi \in \mathcal{F}^{n+1} \mathfrak{g}_B$.
\end{itemize}
\end{definition}

\noindent Homotopy equivalences and gauge homotopy equivalences coincide in most examples:

\begin{proposition} \label{formality-gauge} 
	Let $\C$ be a reduced conilpotent dg coproperad over $R$. 
	\begin{enumerate}[leftmargin=1.2cm]
		\item Suppose that $R$ is a field and that one of the following conditions is satisfied;
		\begin{enumerate} 
			\item[$i.$] $R$ is of characteristic zero;
			\item[$ii.$] $\C$ is a symmetric quasi-planar cooperad.
		\end{enumerate}	$\Cobar \C$-algebras are homotopy equivalent if and only if they are gauge homotopy equivalent. 	
		\item Suppose that $\C$ is non-symmetric operad. 
		\begin{enumerate}
			\item[$i.$] If $R$ is a field, homotopy equivalent $\Cobar \C$-algebras are  gauge homotopy equivalent.
			\item[$ii.$] If $\C$ and $\Cobar \C$ are aritywise and degreewise flat, gauge homotopy equivalent $\Cobar \C$-algebras are homotopy equivalent.
		\end{enumerate}
	\end{enumerate}
\end{proposition}

\begin{proof}\leavevmode
	\begin{enumerate}
		\item The points $i.$ and $ii.$ are respectively due to \cite[Theorem 1.11]{SPH}  and \cite[Proposition~43]{LRL23}.
		\item  The proof is the exact same than the proof of Point (2) in \cite[Propostion~2.21]{Kaledin}, replacing the homology $(H(A), \varphi_*)$ by $(B, \psi)$. \qedhere
	\end{enumerate}
\end{proof}

\begin{definition}[Transferred structure]
A $\Cobar \C$-algebra $(A, \varphi)$ admits a \emph{transferred structure} if there exist a $\Cobar \C$-algebra structure $(H(A), \varphi_t)$ and two $\infty$-quasi-isomorphisms 
	\[\hbox{
		\begin{tikzpicture}
			
			\def\upshift{0.075}
			\def\downshift{0.075}
			\pgfmathsetmacro{\midshift}{0.005}
			
			\node[left] (x) at (0, 0) {$(A,\varphi)$};
			\node[right=1.5 cm of x] (y) {$(H(A),\varphi_t) \ .$};
			
			\draw[->] ($(x.east) + (0.1, \upshift)$) -- node[above]{\mbox{\tiny{$p_{\infty}^A$}}} ($(y.west) + (-0.1, \upshift)$);
			\draw[->] ($(y.west) + (-0.1, -\downshift)$) -- node[below]{\mbox{\tiny{$i_{\infty}^A$}}} ($(x.east) + (0.1, -\downshift)$);
			
	\end{tikzpicture}}
	\] 
\end{definition}

\noindent At least if $R$ is a characteristic zero field, any $\Cobar \C$-algebra $(A, \varphi)$ admits a transferred structure by the homotopy transfer theorem: 

\begin{theorem}[Homotopy transfer theorem] \label{HTT}  
	\noindent	Let $R$ be a commutative ground ring and let $A$ be a chain complex related to $H(A)$ via a contraction, i.e. there exists \[
	\hbox{
		\begin{tikzpicture}
			
			\def\upshift{0.075}
			\def\downshift{0.075}
			\pgfmathsetmacro{\midshift}{0.005}
			
			\node[left] (x) at (0, 0) {$(A,d)$};
			\node[right=1.5 cm of x] (y) {$(H(A),0)$};
			
			\draw[->] ($(x.east) + (0.1, \upshift)$) -- node[above]{\mbox{\tiny{$p$}}} ($(y.west) + (-0.1, \upshift)$);
			\draw[->] ($(y.west) + (-0.1, -\downshift)$) -- node[below]{\mbox{\tiny{$i$}}} ($(x.east) + (0.1, -\downshift)$);
			
			\draw[->] ($(x.south west) + (0, 0.1)$) to [out=-160,in=160,looseness=5] node[left]{\mbox{\tiny{$h$}}} ($(x.north west) - (0, 0.1)$);
	\end{tikzpicture}} \] such that $ip - \mathrm{id}_A = d_A h + h d_A\ ,$  $pi = \mathrm{id}_{H(A)} \ ,$ $h^2 =0\ , $ $  ph = 0\ ,$ $hi =0 \ .   $  Let $\C$ be a reduced conilpotent dg coproperad over $R$ and suppose that we are in one of the three following cases. 
	\begin{itemize}
		\item[$i.$] $\C$ is a symmetric cooperad and $R$ is a $\mathbb{Q}$-algebra;
		\item[$ii.$]  $\C$ is a symmetric quasi-planar cooperad and $R$ is a field;
		\item[$iii.$] $\C$ is a coproperad and $R$ is a characteristic zero field.  
	\end{itemize} 	For any $\Cobar \C$-algebra structure, there exists a transferred structure $(H(A), \varphi_t)$ such that the $\infty$-quasi-isomorphisms $i_{\infty}^A$ and $p_{\infty}^A$ extend the embedding $i$ and the projection $p$. 
\end{theorem}

\begin{proof} In the case $i.$, we refer the reader to \cite{berglund},  \cite[Section~10.3]{LodayVallette12} and references therein. In the case $ii.$ and $iii.$, we refer the reader respectively to \cite[Section~5.5]{LRL23} and \cite[Theorem~4.14]{PHC}. 
\end{proof}

\begin{proposition}\label{isotopy} Let $(A, \varphi)$ and $(B, \psi)$ be two $\Cobar \C$-algebra structure admitting transferred structures. The following propositions are equivalent. 
	\begin{enumerate}
		\item These $\Cobar \C$-algebra structures are gauge homotopy equivalent;
		\item There exists an $\infty$-isomorphism $
		f : (H(A), \varphi_t)  
		\rightsquigarrow (H(B), \psi_t) \ ; $
		\item There exists an isomorphism $\iota : H(B) \to H(A)$ and an $\infty$-isotopy \begin{center}
			\begin{tikzcd}[column sep=normal]
				(H(A), \varphi_t) \ar[r,squiggly,"\sim"] 
				& (H(A), \iota \cdot \psi_t)  \ .
			\end{tikzcd} 
		\end{center} 
		\item There exist an isomorphism $\iota : H(B) \to H(A)$ and $f \in \Gamma_{H(A)}$ such that \[f \cdot \varphi_t = \iota \cdot \psi_t  \ . \]
	\end{enumerate}
\end{proposition}

\begin{proof} \leavevmode
	\begin{description}
		\item[$(1) \Rightarrow (2)$] If $g : (A,\varphi)  \rightsquigarrow (B, \psi) $ is a gauge homotopy equivalence, the composite \[f \coloneqq p_{\infty}^B\circledcirc g \circledcirc i_{\infty}^A\] gives the desired $\infty$-isomorphism. 
		\item[$(2) \Rightarrow (3)$] Let $f$ be the $\infty$-isomorphism given by (2). Let us consider the isomorphism  \[ \iota \coloneqq \left(f^{(0)} \right)^{-1} : H(B) \to H(A) \ . \] By Proposition \ref{Adf properadic0}, $\iota \cdot \psi_t$ is a Maurer--Cartan element  and $\iota :  \psi_t \rightsquigarrow \iota \cdot \psi_t$ is an $\infty$-isomorphism so that $\iota \circledcirc f$ gives the desired $\infty$-isotopy.
		\item[$(3) \Rightarrow (1)$] Let $k$ be the $\infty$-isotopy given by $(3)$. Then the composite \[ i_{\infty}^B\circledcirc \iota^{-1} \circledcirc k  \circledcirc p_{\infty}^A\] gives the desired gauge homotopy equivalence.   
		\item[$(3) \Leftrightarrow (4)$] This is Point (3) of Proposition \ref{Adf properadic0}. \qedhere
	\end{description}
\end{proof}

\begin{corollary}\label{coro isotopy}
	Let $(A, \varphi)$ and $(B, \psi)$ be two $\Cobar \C$-algebra structure admitting transferred structures. The following propositions are equivalent. 
	\begin{enumerate}
		\item These $\Cobar \C$-algebra structures are gauge $n$-homotopy equivalent;
		\item There exist an isomorphism $\iota : H(B) \to H(A)$ and $f \in \Gamma_{H(A)}$ such that \[f \cdot \varphi_t \equiv \iota \cdot \psi_t \pmod{ \mathcal{F}^{n+1} \mathfrak{g}_{H(A)}} \ .\]
	\end{enumerate}
\end{corollary}

\begin{proof}
	This is a direct consequence of the equivalence (1) $\Leftrightarrow $ (4) of Proposition \ref{isotopy}.
\end{proof}

\subsection{Obstruction sequences to homotopy equivalences}\label{2.35} 

 In characteristic zero, the notion of gauge homotopy equivalences can be fully recast as a gauge equivalence problem. This reformulation is crucial, as it allows us to apply the obstruction-theoretic methods developed in Section \ref{section1}. This leads to Theorem \ref{obstru en char 0} that gives obstruction sequences to gauge homotopy equivalences in characteristic zero. This result also set the stage for the obstruction-theoretic approach of formality over any coefficient ring developed in Section \ref{2.4}.

\begin{definition}[Gauge equivalence degree]
	Let $A$ and $B$ be chains complexes with an isomorphism $\iota : H(A) \cong H(B)$. Let $(A, \varphi)$ and $(B, \psi)$ be two $\Cobar \C$-algebra structures admitting transferred structures. Their \emph{gauge equivalence degree} is the one of $\varphi_t$ and $\iota \cdot \psi_t$. 
\end{definition}

\begin{theorem}\label{obstru en char 0}
Let $R$ be a $\mathbb{Q}$-algebra. Let $(A, \varphi)$ and $(B, \psi)$ be two $\Cobar \C$-algebra structures admitting transferred structures.  Suppose that there exists an isomorphism $\iota : H(A) \cong H(B)$.
	\begin{enumerate}
		\item[$(I)$] The following assertions are equivalent. 
		\begin{enumerate}
			\item[$(1)$] The gauge equivalence degree of $\varphi$ and $\psi$ is equal to $\infty$~.
			\item[$(2)$] For all $k \geqslant 1$, the algebras $(A, \varphi)$ and $(B, \psi)$ are gauge $k$-homotopy equivalent. 
		\end{enumerate}   
		\item[$(II)$] The following assertions are equivalent. 
		\begin{enumerate}
			\item[$(3)$] The gauge equivalence degree of $\varphi$ and $\phi$ is equal to $n \in \mathbb{N}$~.
			\item[$(4)$]  The algebras $(A, \varphi)$ and $(B, \psi)$ are gauge $(n-1)$-homotopy equivalent but not gauge $n$-homotopy equivalent. 
		\end{enumerate}	
	\end{enumerate}
\end{theorem}

\begin{proof}
Thanks to Theorem \ref{graphexp2}, this result is a direct application of Theorem \ref{suites d'obstru} using the equivalence (1) $\Leftrightarrow $ (4) of Proposition \ref{isotopy} for Point (1) and Corollary \ref{coro isotopy} for Point (2). 
\end{proof}

\noindent As in Section \ref{section1}, one can highlight sufficient conditions on the complete dg Lie algebra ensuring the gauge homotopy equivalence between $\varphi$ and $\psi$.

\begin{theorem}[The bounded case]\label{the bounded case}
Under the assumption of Theorem \ref{obstru en char 0}, suppose that $\mathfrak{g}_{H(A)}$ is a bounded dg Lie algebra. The following assertions are equivalent. 
	\begin{enumerate}
			\item The gauge equivalence degree of $\varphi$ and $\psi$ is equal to $\infty$~.
			\item The algebras $(A, \varphi)$ and $(B, \psi)$ are gauge homotopy equivalent. 
	\end{enumerate}
	Furthermore, if there exists an $\Omega \C$-algebra structure $(H(A), \phi)$ such that \[H_{-1} \left(\mathfrak{g}_{H(A)}^{\phi} \right) = 0 \] all the $\Omega \C $-algebra structures on $H(A)$ are gauge homotopy equivalent.   
\end{theorem}

\begin{proof}
	Thanks to Theorem \ref{graphexp2}, this result is a direct application of Theorem \ref{suites d'obstructions3} using the equivalence (1) $\Leftrightarrow $ (4) of Proposition \ref{isotopy}. 
\end{proof}

\begin{example}
	In the case of an highly connected variety $M$, the dg Lie algebra $\mathfrak{g}_{H(A)}$ corresponding to $A = \Omega^*_{dr}(M)$ is bounded, see Section \ref{section3}.
\end{example}

\begin{theorem}[The weight-graded case]\label{the weight case}
	Under the assumption of Theorem \ref{obstru en char 0}, suppose that  for $\delta \geqslant 1$, the dg Lie algebra $\mathfrak{g}_{H(A)}$ is  $\delta$-weight-graded,  \[\psi_t \in \mathfrak{g}_{H(A)}^{(\delta)} \quad \mbox{and} \quad  \varphi_t - \psi_t \in \mathcal{F}^{\delta +1} \mathfrak{g}_{H(A)} \ .\] The following assertions are equivalent. 
	\begin{enumerate}
		\item The gauge equivalence degree of $\varphi$ and $\psi$ is equal to $\infty$~.
		\item The algebras $(A, \varphi)$ and $(B, \psi)$ are gauge homotopy equivalent. 
	\end{enumerate}
	Furthermore, if there exists an $\Omega \C$-algebra structure $\phi \in \mathrm{MC}\left(\mathfrak{g}_{H(A)}^{(\delta)}\right)$  such that \[\mathcal{F}^{\delta +1} H_{-1}(\mathfrak{g}^{\phi}) = 0 \ .\] Then, any $\Omega \C$-algebra structure $(H(A),\phi')$ such that $\phi' - \phi \in \mathcal{F}^{\delta +1} \mathfrak{g}$ is gauge equivalent to $\phi$. 
\end{theorem}

\begin{proof}
	Thanks to Theorem \ref{graphexp2}, this result is a direct application of Theorem \ref{suites d'obstructions} using the equivalence (1) $\Leftrightarrow $ (4) of Proposition \ref{isotopy}. 
\end{proof}

\begin{remark}
	In this weight-graded case, obstruction sequences can be generalized to arbitrary coefficient rings, see Section \ref{2.4}. 
\end{remark}

\begin{theorem}[Equivalence descent]\label{eq descent}
	Under the assumption of Theorem \ref{obstru en char 0}, let $S$ be a faithfully flat commutative $R$-algebra. The algebras $(H(A), \varphi_t)$ and $(H(B), \psi_t)$ are $\infty$-isotopic if and only if the algebras $(H(A) \otimes S, \varphi_t \otimes 1)$ and $(H(B) \otimes S, \psi_t \otimes 1)$ are $\infty$-isotopic. 
\end{theorem}

\begin{proof}
	Choose an $R$-linear retraction of the inclusion $R \hookrightarrow S$. It induces a morphism
	\[
	r \colon \mathfrak{g}_{H(A)\otimes S} \longrightarrow \mathfrak{g}_{H(A)}
	\]
	which is a retraction of the natural inclusion	in the category of filtered $\mathfrak{g}_{H(A)}$-modules.	By Theorem~\ref{graphexp2}, the result then follows from the same argument as in Theorem~\ref{cas complet}. Indeed, although $\mathfrak{g}_{H(A)\otimes S}$ need not be canonically identified with \[\mathfrak{g}_{H(A)} \otimes_R S \ ,\] the proof of Theorem~\ref{cas complet} only requires a compatible projection onto $\mathfrak{g}_{H(A)}$. Therefore, whenever such a projection is needed, we may use the retraction $r$ and the proof of Theorem~\ref{cas complet} carries over verbatim.
\end{proof}

\subsection{Obstruction sequences to gauge formality}\label{2.4}
This section focuses on the special case of formality. In this situation, the construction adapts to more general coefficient ring $R$. If $R$ is a $\mathbb{Q}$-algebra, one can already apply Theorem \ref{obstru en char 0}. Let us suppose that $R$ is not a $\mathbb{Q}$-algebra and let $\mathfrak{n}$ be the smallest integer not invertible in $R$. Suppose that $\C$ is a reduced weight-graded dg coproperad over $R$, i.e. the exhaustive filtration comes from a weight-grading \[\C =   \mathcal{I} \oplus \C^{(1)} \oplus \C^{(2)} \oplus  \cdots \oplus \C^{(n)} \oplus \cdots \quad \ , \quad d_{\C} \left(\C^{(k)}\right) \subset \C^{(k - 1)} \ . \]

\begin{example} Suppose that $A$ has no differential. 
	\begin{itemize}
		\item[$\centerdot$] The algebra $\mathfrak{g}_A$ is a $1$-weight-graded dg Lie algebra. If furthermore $d_{\C}$ is trivial, then $\mathfrak{g}_A$ is a $\delta$-weight-graded dg Lie algebra for all $\delta$, see Definition \ref{delta wg}. 
		\item[$\centerdot$] If  $\psi \in \mathfrak{g}_A$ is concentrated in weight $1$, then $\mathfrak{g}_A^{\psi}$ is a $1$-weight-graded dg Lie algebra. 
		\item[$\centerdot$] If furthermore $d_{\C}$ is trivial and $\psi$ is concentrated in weight $\delta$, then $\mathfrak{g}_A^{\psi}$ is a $\delta$-weight-graded dg Lie algebra. 
	\end{itemize}
\end{example}

\noindent Let us consider the properad 
\[\P \coloneqq \mathcal{T}\left(s^{-1}\C^{(1)}\right) / \left(d_{\Cobar \C} \left(s^{-1}\C^{(2)}\right)\right) \ , \] which comes with a twisting morphism $\C \to \P$. By \cite[Proposition~3.11]{Kaledin}, any $\Cobar \C$-algebra structure $(A, \varphi)$ induces a canonical $\P$-algebra structure \[(H(A),\varphi_*) \ . \]

\begin{definition}[Gauge formality] \label{definition formality prop}
	A $\Cobar \C$-algebra structure $(A, \varphi)$ is said to be 
	\begin{itemize}
		\item[$\centerdot$] \emph{gauge formal} if it is gauge homotopy equivalent to $(H(A), \varphi_*)$;
		\item[$\centerdot$] \emph{gauge $n$-formal} if it is gauge $(n+1)-$homotopy equivalent to $(H(A), \varphi_*)$. 
	\end{itemize} 
\end{definition}

	\begin{remark}
	For the sake of clarity, we deal with gauge formality that is with the case $\delta = 1$. But if $d_{\C} = 0$ then $\mathfrak{g}_H$ is also $\delta$-weight-graded dg Lie algebra for all $\delta \geqslant 1$. All what follows adapts directly to detect if an $\Omega \C$-algebra structure $\varphi \in  \mathcal{F}^{\delta}(\mathfrak{g}_H)$ is related to $\varphi^{\delta}$ by an $\infty$-isotopy. 
\end{remark}

\begin{proposition}\label{obstru bis}
Let $\varphi$ be an $\Cobar \C$-algebra structure on a graded $R$-module $H$. Let us set \[\psi \coloneqq \varphi^{(1)} \quad \mbox{and} \quad \mathfrak{h} \coloneqq \mathfrak{g}_H^{\psi} \] and suppose that $ \varphi - \psi \in \mathcal{F}^{n} \mathfrak{h}$ where $n \geqslant 1$ is such that $(n-1)!$ is invertible in $R$. Let us consider \[\vartheta_{n} \coloneqq \left[\varphi^{(n)} \right] \in H_{-1}\left( \mathfrak{h} / \mathcal{F}^{n+1}\mathfrak{h}  \right) \ .\] The following assertions are equivalent.
\begin{enumerate}
	\item The homology class $\vartheta_{n}$ vanishes; 
	\item The $\Cobar \C$-algebra structures $\varphi$ and $\psi$ are gauge $n$-homotopy equivalent.  
\end{enumerate}
\end{proposition}

\begin{proof} $\leavevmode$
\begin{description}
	\item[$(1) \Rightarrow  (2)$] Let us suppose that $\vartheta_{n} = 0$~. There exists $\nu \in \mathfrak{g}_H $ of degree zero such that \[ \varphi^{(n)} = d^{\psi} (\nu)  \ .\] The element $\lambda \coloneqq \nu^{(n -1)}$ also satisfies $ \varphi^{(n)} = d^{\psi} \left(\lambda\right)$~. The $\infty$-isotopy defined by $f \coloneqq 1 + \lambda$ and the Maurer--Cartan element $\phi \coloneqq f \cdot \varphi$ satisfy the desired properties by \cite[Lemma~3.22]{Kaledin}.

		\item[$(2) \Rightarrow (1)$] Suppose that the structures are gauge $n$-homotopy equivalent with $\phi$ an $\Omega \C$-algebra structure such that $\phi - \psi \in \mathcal{F}^{n+1} \mathfrak{g}_H $ and $f : \varphi \rightsquigarrow \phi$ an $\infty$-isotopy. Let us prove that $\vartheta_{n} = 0$~, i.e. that there exists $\nu \in \mathfrak{g}_H $ of degree zero such that \[ \varphi^{(n)} = d^{\psi} (\nu)  \ .\] If $R$ is a $\mathbb{Q}$-algebra, by Theorem \ref{graphexp2} and Proposition \ref{obstru}, the gauge $\nu \coloneqq \mathrm{log}(f)$ and even $\nu \coloneqq \mathrm{log}(f)^{(n-1)}$ gives the desired result. Over any ring $R$ such that $(n-1)!$ is invertible in $R$, we claim that the component $\nu \coloneqq \mathrm{log}(f)^{(n-1)}$ is still well-defined and gives the desired result. Let us first prove that this component is well-defined. In the proof of Theorem \ref{graphexp2}, the authors prove by induction on the weight that the graph exponential map is surjective and thus construct the logarithm. To get from weight $k-1$ to weight $k$, they add the action of a finite sum of directed simple graphs with at most $k$ vertices on the components of lower weight. This added action is well defined if $k !$ is invertible in $R$. Let us now prove that $\nu \coloneqq \mathrm{log}(f)^{(n-1)}$ gives the desired result. This can be achieved in two different ways. One direct way is observing that the proof of the implication $(2) \Rightarrow (1)$ of Proposition \ref{graphexp2} holds modulo $\mathcal{F}^{n+1} \mathfrak{g}_H$ over $R$ since one is not dividing by any integer greater than $(n-1)!$. Otherwise, since \[\varphi - \psi \in \mathcal{F}^n \mathfrak{h} \quad \mbox{and} \quad f \circledcirc f^{-1} = 1 \ ,\] the induction on $k$ can be carried out again to prove that for all $1 \leqslant k \leqslant n $,    \[d \upsilon \in  \mathcal{F}^{k} \mathfrak{h}  \quad \mbox{and} \quad \varphi^{(n)}  \equiv d \upsilon  \pmod{\mathcal{F}^{n+1} \mathfrak{h}} \ . \] This implies that $\vartheta_{n} = 0$. 	    
 \qedhere
\end{description}
\end{proof}

\begin{construction} \label{constru bis}
Let $\varphi$ be an $\Cobar \C$-algebra structure on a graded $R$-module $H$. We would like to detect whether $\varphi$ and $\psi \coloneqq \varphi^{(1)}$ are gauge homotopy equivalent. Let us set \[\vartheta_2 \coloneqq \left[ \varphi^{(2)}\right] \in H_{-1}\left( \mathfrak{h} / \mathcal{F}^{3} \mathfrak{h}  \right) . \]  If $2 \leqslant \mathfrak{n}$, we have the two following cases. 
	\begin{itemize}
		\item[$\centerdot$] If $\vartheta_2 \neq 0$~, then $\varphi$ and $\psi$ are not gauge homotopy equivalent, by the implication $(2) \Rightarrow (1)$ of Proposition~\ref{obstru bis}. 
		\item[$\centerdot$] If $\vartheta_2 = 0$~, then $\varphi$ and $\psi$ are gauge $2$-homotopy equivalent, by the implication $(1) \Rightarrow (2)$ of Proposition \ref{obstru bis}. More precisely, $\varphi$ is gauge homotopy equivalent to an $\Cobar \C$-algebra structure $(H, \phi_3)$ such that $ \phi_3 - \psi \in  \mathcal{F}^{3} \mathfrak{h}$.
	\end{itemize}
	If $\vartheta_2 = 0$~, we consider \[\vartheta_3 \coloneqq \left[ \phi_3^{(3)}\right] \in H_{-1}\left( \mathfrak{h} / \mathcal{F}^{4} \mathfrak{h}  \right) . \] If $3 \leqslant \mathfrak{n}$, we have the two following cases. 
	\begin{itemize}
		\item[$\centerdot$] If $\vartheta_3 \neq 0$~, then $\varphi$ and $\psi$ are not gauge homotopy equivalent, by the implication $(2) \Rightarrow (1)$ of Proposition~\ref{obstru bis}. 
		\item[$\centerdot$] If $\vartheta_3 = 0$~, then $\varphi$ and $\psi$ are gauge $3$-homotopy equivalent, by the implication $(1) \Rightarrow (2)$ of Proposition \ref{obstru bis}. More precisely, $\phi_3$ and thus $\varphi$ is gauge homotopy equivalent to an $\Cobar \C$-algebra structure $(H, \phi_4)$ such that $ \phi_4 - \psi \in  \mathcal{F}^{4} \mathfrak{h}$.
	\end{itemize}
		If $\vartheta_3 = 0$~, we consider \[\vartheta_4 \coloneqq \left[ \phi_4^{(4)}\right] \in H_{-1}\left( \mathfrak{h} / \mathcal{F}^{5} \mathfrak{h}  \right) . \]  Once again, if $4 \leqslant \mathfrak{n}$,  the following assertions hold by Proposition \ref{obstru bis}. 
	\begin{itemize}
		\item[$\centerdot$] If $\vartheta_4 \neq 0$~, then $\varphi$ and $\psi$ are not gauge homotopy equivalent. 
		\item[$\centerdot$] If $\vartheta_4 = 0$~, then $\varphi$ and $\psi$ are gauge $4$-homotopy equivalent.
	\end{itemize}

	\noindent The construction of such obstruction classes can be performed higher up in a similar way until reaching a non-zero class or until the level $\mathfrak{n}$. This leads to a sequence of classes $\left(\vartheta_k \right)_{1 \leqslant k \leqslant n}$ which is either 
	\begin{itemize}

		\item[$\centerdot$]  a sequence of trivial classes that ends on a nonzero class $\vartheta_{n}$~, with $n \leqslant \mathfrak{n}$, or 
		\item[$\centerdot$]a sequence of trivial classes that ends on a class $\vartheta_{\mathfrak{n}+1}$ (which vanishes or not) when $n = \mathfrak{n}+1$~.
	\end{itemize}  Any such sequence is not unique and depends on the choices made at each level. 
\end{construction}

\begin{definition}[Homotopy equivalence obstruction sequence]\label{gauge triviality sequ bis}
	A \emph{homotopy equivalence obstruction sequence} associated to an $\Cobar \C$-algebra structures $\varphi $ and its first component $\varphi^{(1)}$ is an obstruction sequence $\left(\vartheta_k \right)_{1 \leqslant k \leqslant n}$~, for $n\leqslant \mathfrak{n}+1$~, obtained through Construction \ref{constru bis}. 
\end{definition}

\begin{lemma}\label{gauge deg bis}  
	Let $\left(\vartheta_k \right)_{1 \leqslant k \leqslant n}$ be a homotopy equivalence obstruction sequence of $\varphi$ and $\psi$~, with $n\leqslant \mathfrak{n}+1$. For any other homotopy equivalence obstruction sequence $\left(\vartheta_k' \right)_{1 \leqslant k \leqslant m}$, we have $m=n$~. 
\end{lemma}

\begin{proof}
	The proof uses Proposition \ref{obstru bis} in the way similar to the proof of Lemma \ref{gauge deg}. In this case, we compose $\infty$-isomorphisms through the product $\circledcirc$ instead of composing gauges through the $\mathrm{BCH}$ formula. 
\end{proof}

\begin{definition}[Gauge triviality degree]
Given an $\Cobar \C$-algebra structures $\varphi $, the element $n\leqslant \mathfrak{n}+1$ which indices the last class of a homotopy equivalence obstruction sequence associated to $\varphi $ and its first component $\varphi^{(1)}$ is called the \emph{gauge triviality degree} of $\varphi$. 
\end{definition}

\begin{theorem}\label{suites d'obstructions2}
Let $\varphi$ be an $\Cobar \C$-algebra structure on a graded $R$-module $H$. 
	\begin{enumerate}
	\item[(I)] The following assertions are equivalent. 
	\begin{enumerate}
		\item[(1)] The gauge triviality degree of $\varphi$ is equal to $ \mathfrak{n} + 1$~.
		\item[(2)] The $\Cobar \C$-algebra structures $\varphi$ and $\varphi^{(1)}$ are gauge  $\mathfrak{n}$-homotopy equivalent.  
	\end{enumerate}   
	\medskip
	\item[(II)] The following assertions are equivalent. 
	\begin{enumerate}
	\item[(1)] The gauge triviality degree of $\varphi$ is equal to $n \leqslant \mathfrak{n}$~.
	\item[(2)] The $\Cobar \C$-algebra structures $\varphi$ and $\varphi^{(1)}$ are gauge  $(n-1)$-homotopy equivalent but not gauge $n$-homotopy equivalent.  
\end{enumerate} 
	\end{enumerate}
\end{theorem}

\begin{proof}
	The proof is the same one that the proof of Theorem \ref{suites d'obstru}, using Proposition \ref{obstru bis} instead of Proposition \ref{obstru} and composing $\infty$-isotopies through the product $\circledcirc$ instead of composing gauges through the $\mathrm{BCH}$ formula. 	
\end{proof}

\begin{definition}
 Let $(A, \varphi)$ be an $\Cobar \C$-algebra structure admitting a transferred structure. Its \emph{gauge triviality degree} is the gauge triviality degree of $\varphi_t$. 
\end{definition}

\begin{theorem}\label{obstru en char p}
	Let $R$ be a ring which is not a $\mathbb{Q}$-algebra. Let $\mathfrak{n}$ be the smallest integer not invertible in $R$.  Let $(A, \varphi)$ be an $\Cobar \C$-algebra structure admitting a transferred structure. 
	\begin{enumerate}
		\item[$(I)$] The following assertions are equivalent. 
		\begin{enumerate}
			\item[$(1)$] The gauge triviality degree of $\varphi$ is equal to $ \mathfrak{n} + 1$~.
			\item[$(2)$] The $\Cobar \C$-algebra structure $(A, \varphi)$  is gauge  $(\mathfrak{n-1})$-formal.  
		\end{enumerate}   
		\item[$(II)$] The following assertions are equivalent. 
		\begin{enumerate}
			\item[$(3)$] The gauge triviality degree of $\varphi$ is equal to $n \in \mathbb{N}$~.
			\item[$(4)$]  The $\Cobar \C$-algebra structure $(A, \varphi)$  is gauge  $(n-2)$-formal but not gauge  $(n-1)$-formal. 
		\end{enumerate}	
	\end{enumerate}
\end{theorem}

\begin{proof}
	This result is a direct application of Theorem \ref{suites d'obstructions2} using Definition \ref{definition formality prop}.
\end{proof}

\subsection{Interlude : Properadic Kaledin classes over any coefficient ring}\label{2.5}
In \cite[Section~3]{Kaledin}, we construct properadic Kaledin classes over a characteristic zero field studying a particular case of homotopy equivalences called \emph{gauge formality}. As explained in this section, this construction generalizes to any commutative ground ring as in the operadic case. We also make precise the like between homotopy equivalence obstruction sequences and Kaledin classes in the formality setting. Let $\C$ be a weight-graded coproperad.

\begin{definition}[Properadic Kaledin classes]
	Let $\varphi \in \mathrm{MC}(\mathfrak{g}_H)$ be an $\Cobar \C$-algebra structure on a graded $R$-module $H$ and let \[\mathfrak{D}(\varphi) \coloneqq  \varphi^{(1)} + \varphi^{(2 )} \hbar + \varphi^{(3)} \hbar^2 + \cdots \] be its prismatic decomposition. The \emph{Kaledin class} $K_{\varphi}$ is the Kaledin class of $\mathfrak{D}(\varphi)$, i.e. \[K_{\varphi} =  \left[\partial_{\hbar}   \mathfrak{D}(\varphi) \right] \in H_{-1}\left(\mathfrak{g}_H[\![\hbar]\!]^{\Phi}\right) \ . \] Its \emph{$n^{\text{th}}$-truncated Kaledin class} $K_{\varphi}^n$ is the $n^{\text{th}}$-truncated Kaledin class of $\mathfrak{D}(\varphi)$, i.e. \[ K^n_{\varphi} = \left[\varphi^{(2)} + 2 \varphi^{(3)} \hbar + \dots + n \varphi^{(n+1)} \hbar^{n-1} \right]  \in H_{-1}\left(\left(\mathfrak{g}_H[\![\hbar]\!]/(\hbar^{n })\right)^{\overline{\Phi}^n}\right) \ .  \] 
\end{definition}

\begin{remark}
	Let $\varphi$ be an $\Cobar \C$-algebra structure on a graded $R$-module $H$.  Let us set $\mathfrak{h} \coloneqq \mathfrak{g}^{\varphi^{(1)}}$. If $ \varphi - \varphi^{(1)} \in \mathcal{F}^n \mathfrak{h}$, for $n > 1$, consider the obstruction class  \[\vartheta_{n} \coloneqq \left[\varphi^{(n)} \right] \in H_{-1}\left( \mathfrak{h} / \mathcal{F}^{n+1}\mathfrak{h}  \right) \ ,\] given by Proposition \ref{obstru bis}. It corresponds to the Kaledin class  \[ K^{n-1}_{\varphi} = \left[ (n - 1) \varphi^{(n)} \hbar^{n-2} \right]  \in H_{-1}\left(\mathfrak{h}[\![\hbar]\!]/(\hbar^{n -1 })\right) \ .  \] 
\end{remark}

\begin{proposition}\label{Kaledin prop} 	Let $R$ be a commutative ground ring and let $\C$ be a reduced weight-graded dg cooperad over $R$. Let $\varphi \in \mathrm{MC}(\mathfrak{g}_H)$ be an $\Cobar \C$-algebra structure on a graded $R$-module $H$~. 
	\begin{enumerate}
		\item  Let $n \geqslant 1$ be an integer such that $n !$ is a unit in $R$. The $n^{\text{th}}$-truncated Kaledin class $K^n_{\varphi}$  is zero if and only if there exists an $\infty$-isotopy $f : \varphi \rightsquigarrow  \psi $ where $\psi \in \mathrm{MC}(\mathfrak{g}_H)$ is such that $\psi^{(k)} = 0$~, for $2 \leqslant k \leqslant n + 1 $. 
		
		\item If $R$ is a $\QQ$-algebra, the Kaledin class  $K_{\varphi}$ vanishes if and only if there exists an $\infty$-isotopy $f : \varphi \rightsquigarrow \varphi^{(1)} $~.
	\end{enumerate}
\end{proposition}

\begin{definition}
	The \emph{prismatic decomposition} of $f  \in \mathfrak{a}_H$ in degree zero is defined by \[\mathfrak{D} \left(f\right) \coloneqq  f^{(0)} + f^{(1)}\hbar +  f^{(2)}\hbar^2 + \cdots \in \mathfrak{a}_H[\![\hbar]\!] \ .\] 
\end{definition}

\begin{remark}
	The pre-Lie product $\star$, the circle product $\circledcirc$ and the left and right actions $\lhd$ and $\rhd$ extend by $\hbar$-linearity to $\mathfrak{a}_H[\![\hbar]\!]$.  Furthermore, if $f : \varphi \rightsquigarrow \psi$ is an $\infty$-isomorphism then  \[\mathfrak{D} \left(f\right) \rhd \mathfrak{D} \left( \varphi \right) -  \mathfrak{D} \left( \psi \right) \lhd  \mathfrak{D} \left(f\right) =  d \left( \mathfrak{D} \left(f\right) \right) \ .\]
\end{remark}

\begin{lemma} Let $n \geqslant 1$ and let $f : \varphi \rightsquigarrow \psi$ be an $\infty$-isomorphism. Let us denote $\Phi \coloneqq \mathfrak{D}(\varphi)$~, $\Psi \coloneqq \mathfrak{D}(\psi)$ and $F \coloneqq \mathfrak{D}(f)$. There are isomorphisms of dg Lie algebras \[\mathrm{Ad}_F : \mathfrak{g}_H[\![\hbar]\!]^{\Phi} \to \mathfrak{g}_H[\![\hbar]\!]^{\Psi} \quad \mbox{and} \quad \mathrm{Ad}_{F}^n :  \left(\mathfrak{g}_H[\![\hbar]\!] /(\hbar^{n}) \right)^{\overline{\Phi}^n} \to  \left(\mathfrak{g}_H[\![\hbar]\!] /(\hbar^{n}) \right)^{\overline{\Psi}^n}   \] induced by the mapping $x \mapsto (F ; F \star x) \circledcirc F^{-1}$.
\end{lemma}

\begin{proof}
	The proof of Proposition \ref{Adf properadic} holds \emph{mutatis mutandis} by $\hbar$-linearity. 
\end{proof}

\begin{lemma}[Invariance of Kaledin classes under $\infty$-isomorphisms]\label{piquant prop} Let $f : \varphi \rightsquigarrow \psi$ be an $\infty$-isomorphism. Let us set $\Phi \coloneqq \mathfrak{D}(\varphi)$~, $\Psi \coloneqq \mathfrak{D}(\psi)$ and $F \coloneqq \mathfrak{D}(f)$. We have \[ \Psi = 	\mathrm{Ad}_{F} ( \Phi) - \left(F ; d(F)  \right) \circledcirc F^{-1} \ . \]
	\begin{enumerate}
		\item The classes $K_{\Psi}$ and $\mathrm{Ad}_{F}\left(K_{\Phi}\right)$ are equal in \[H_{-1}\left(\mathfrak{g}_H[\![\hbar]\!]^{\Psi}\right) \ ,\] where we still denote by $\mathrm{Ad}_{F}$ the induced isomorphism on homology. 
		\item For all $n \geqslant 1$, the classes $K_{\Psi}^n$ and $\mathrm{Ad}_{F}^n\left(K_{\Phi}^n\right)$ are equal in \[H_{-1}\left(\left(\mathfrak{g}_H[\![\hbar]\!] /(\hbar^{n}) \right)^{\overline{\Psi}^n}\right) \ ,\] where we still denote by $\mathrm{Ad}_{F}^n$ the induced isomorphism on homology. 
	\end{enumerate}
\end{lemma}

\begin{proof}
Let us prove point $(1)$. More precisely, let us prove that the following equality holds \begin{equation*}
	\mathrm{Ad}_{F} (\partial_{\hbar} \Phi) - \partial_{\hbar} \Psi  = d^{\Psi} \left( \left(F ; \partial_{\hbar} F \right) \circledcirc F^{-1}  \right) \ .
\end{equation*} First of all, the differentiation $\partial_{\hbar}$ is a coderivation and satisfies the sames compatibility relation than the differential $d$. In particular, using the same argument than in the proof of Point (3) of Proposition \ref{Adf properadic0}, we have \[ \mathrm{Ad}_{F} (\partial_{\hbar} \Phi) - \partial_{\hbar} \left( \mathrm{Ad}_{F} (\Phi) \right) =    \left[\mathrm{Ad}_{F} (\Phi), \left(F ; \partial_{\hbar} F \right) \circledcirc F^{-1} \right] \ . \] It suffices to prove that \[ \partial_{\hbar} \left(\left(F ; d(F)  \right) \circledcirc F^{-1}\right)  =  d\left( \left(F ; \partial_{\hbar} F \right) \circledcirc F^{-1}\right) - \left[\left(F ; d(F)  \right) \circledcirc F^{-1}, \left(F ; \partial_{\hbar} F \right) \circledcirc F^{-1} \right] \ .  \] The difference $\partial_{\hbar} \left(\left(F ; d(F)  \right) \circledcirc F^{-1}\right) - d\left( \left(F ; \partial_{\hbar} F \right) \circledcirc F^{-1}\right)$ boils down to \[  (F ; d(F)) \circledcirc (F^{-1} ; \partial_{\hbar} \left( F^{-1} \right) ) - (F ; \partial_{\hbar}F) \circledcirc (F^{-1} ; d \left( F^{-1} \right) ) \ . \] Applying twice formula (8) of Lemma \ref{cal}, leads to \begin{equation*}                                  
		(F ; d(F)) \circledcirc (F^{-1} ; \partial_{\hbar} \left( F^{-1} \right) )  =  \left(\left(F ; d(F)  \right) \circledcirc F^{-1} \right) \star \left( F  \circledcirc (F^{-1} ; \partial_{\hbar} \left( F^{-1} \right) \right) \ .
         \end{equation*}  \begin{equation*}                             
(F ; \partial_{\hbar} (F)) \circledcirc (F^{-1} ; d \left( F^{-1} \right) ) =   \left(\left(F ; \partial_{\hbar}F \right) \circledcirc F^{-1} \right) \star \left( F \circledcirc \left(F^{-1} ; d \left( F^{-1} \right) \right) \right) \ . \end{equation*}  By Point (2) of Lemma \ref{cal2}, we have  $F \circledcirc \left(F^{-1} ; d \left( F^{-1} \right) \right) = -  \left(F ; d \left( F \right) \right) \circledcirc F^{-1}$. The formula \[ d \left(F^{1} \circledcirc F \right) = (F^{-1} ; d \left( F^{-1} \right)) \circledcirc F +  F^{-1} \circledcirc (F; d(F))  \] implies that $F  \circledcirc (F^{-1} ; \partial_{\hbar} \left( F^{-1} \right) = - \left(F ;  \partial_{\hbar} \left( F \right) \right) \circledcirc F^{-1} $.
This concludes the proof. 
\end{proof}

\begin{proof}[Proof of Proposition \ref{Kaledin prop}] 
	
	Let us prove Point (1). Let $\psi \in \mathrm{MC}(\mathfrak{g}_H)$ be such that $\psi^{(k)} = 0$~, for $2 \leqslant k \leqslant n+1$~, and let $\varphi \rightsquigarrow \psi$ be an $\infty$-isotopy. Point (2) of Lemma \ref{piquant prop} implies that \[K_{\Phi}^n = \left(\mathrm{Ad}_{F}^n \right)^{-1} \left(K_{\Psi}^n\right) = \left(\mathrm{Ad}_{F}^n \right)^{-1} \left(K_{\varphi^{(1)}}^n\right) = 0 \ . \] Let us prove the converse result by induction on $n \geqslant 1$. Suppose that \[ K^1_{\Phi} = 0 \in H_{-1}\left(\mathfrak{g}^{\varphi^{(1)}}_H\right) \ . \] There exists $\upsilon \in \mathfrak{g}_H$ of degree zero such that $ d^{\varphi^{(1)}} (\upsilon) =\varphi^{(2)} \ .$ Let us set $f \coloneqq 1 + \upsilon^{(1)}$ and  $\psi \coloneqq f \cdot \varphi$.  By \cite[Lemma~3.22]{Kaledin}, which holds over any coefficient ring, we have \[
	\psi^{(2)}  = \varphi^{(2)} - d^{\varphi^{(1)}} \left(\upsilon^{(1)}\right) = 0 \ .  \] Suppose that $n > 1$ and that the result holds for $n-1$~. If the class $K_{\Phi}^n$ is zero, so does the class $K_{\Phi}^{n-1}$. By the induction hypothesis, there exist $\psi \in \mathrm{MC}(\mathfrak{g}_H)$ such that $\psi^{(k)} = 0$ for $2 \leqslant k \leqslant n$ and an $\infty$-isotopy $g : \varphi \rightsquigarrow \psi$~. Point (2) of Lemma \ref{piquant prop} implies that \[\mathrm{Ad}_{\mathfrak{D}(g)}^n(K^{n}_{\Phi})= K^{n}_{\mathfrak{D}(\psi)} = [n \psi^{(n+1)} \hbar^{n-1}] = 0 \ .\] There exists $\upsilon \coloneqq \upsilon_0 + \upsilon_1 h + \dots + \upsilon_{n-1}\hbar^{n-1} \in \mathfrak{g}_H[\![\hbar]\!]$ of degree zero such that \[d^{\varphi^{(1)}}(\upsilon) \equiv n \psi^{(n+1)} \hbar^{n-1} \pmod{\hbar^{n}} \ .\] Looking at the coefficient of $\hbar^{n-1}$ and in weight $n +1$ on both sides, we have\[ d^{\varphi^{(1)}} \left( \upsilon_{n-1}^{(n)} \right) = n \psi^{(n+1)}  \ . \]  Let us consider $\lambda \coloneqq \frac{1}{n} \upsilon_{n-1}^{(n)}$ and $f \coloneqq 1 + \lambda$~. We set $\psi' \coloneqq f \cdot \psi$. By construction, $f \circledcirc g$ determines an $\infty$-isotopy $ \varphi \rightsquigarrow  \psi' \ .$ The result \cite[Lemma~3.22]{Kaledin} implies that \[\left(\psi'\right)^{(k)} = 0~, \quad \mbox{for} \; 2 \leqslant k \leqslant n~, \quad \mbox{and} \quad
	(\psi')^{(n+1)}  = \psi^{(n+1)} - d^{\varphi^{(1)}} \left( \lambda\right) = 0 \ . \]  The proof of Point (2) is similar to the one of \cite[Proposition~2.32 (2)]{Kaledin}. 
\end{proof}

\begin{theorem}\label{B prop}
	Let $R$ be a commutative ring. Let $\C$ be a reduced weight-graded dg cooperad over $R$. Let $(A, \varphi)$ be a $\Cobar \C$-algebra structure such that there exists a transferred structure $(H(A),\varphi_t)$.
	\begin{enumerate}
		\item Let $n \geqslant 1$ be an integer such that $n !$ is a unit in $R$. The algebra $(A, \varphi)$ is gauge $n $-formal if and only if the truncated class $K^n_{\varphi_t}$ is zero. 
		
		\item If $R$ is a $\mathbb{Q}$-algebra, the algebra $(A, \varphi)$ is gauge formal if and only if the class $K_{\varphi_t}$ is zero.
	\end{enumerate}\end{theorem}

\begin{proof}
	This follows from Proposition \ref{Kaledin prop}, by using Definition \ref{definition formality prop} and Proposition \ref{isotopy}.
\end{proof}

\subsection{Obstruction sequences in the colored setting}\label{2.7}

The obstruction theory described so far admits a natural extension to the setting of colored operads. In this broader context, each input and output of an operation is assigned a label (a “color”), and compositions are only defined when these labels match appropriately. To incorporate additional symmetries, one can work with operads colored not merely by a set but by a groupoid $\mathbb{V}$. A systematic study of the Koszul duality theory for this groupoid-colored situation was carried out by Ward in \cite{War19}; see also \cite[Section 5]{RL2}.

\begin{example}[Endomorphism operad]
	If $A$ is a $\mathbb{V}$-module, i.e. a functor $A : \mathbb{V} \to \mathrm{dgMod}_R$, one can associate to it a $\mathbb{V}$-colored endomorphism operad. For objects $v_0, v_1, \dots, v_r$  in $\mathbb{V}$, the corresponding space of operations is
	$
	\End_A(v_0; v_1,\dots,v_r) := \Hom\bigl(A(v_0),\, A(v_1)\otimes \cdots \otimes A(v_r)\bigr),
	$
	where the symmetric group $\mathbb{S}_r$ acts by permuting the tensor factors.
\end{example}

\noindent Throughout this section, we fix a groupoid $\mathbb{V}$ together with a reduced conilpotent $\mathbb{V}$-colored dg cooperad $\C$ over $R$. 
The results of Sections \ref{2.35}, \ref{2.4}, and \ref{2.5} extend directly to this setting by working in the following colored version of the convolution dg Lie algebra.

\begin{definition} [Colored convolution dg pre-Lie algebra] The \emph{convolution dg pre-Lie algebra} associated to $\C$ and to a $\mathbb{V}$-module $A$ is the dg pre-Lie algebra \[\mathfrak{g}_{A} := \left( \Hom_{\mathbb{S}} \left(\overline{\C}, \mathrm{End}_A\right), \star, d \right) \ , \] where the underlying space is \[\Hom_{ \mathcal{S}_{\mathbb{V}}} \left(\overline{\C}, \End_A\right) : = \prod_{ \mathcal{S}_{\mathbb{V}}} \Hom_{\mathbb{S}} \left(\overline{\C}\left(v_0;v_1, \dots , v_r\right), \End_A\left(v_0; v_1, \dots , v_r\right)\right)\ ,\] where $\mathcal{S}_{\mathbb{V}} \coloneqq \lbrace \left(v_1, \dots , v_r; v_0\right) \in \mathrm{ob} \left(\mathbb{V}\right)^{r+1} \rbrace$. It is equipped with the pre-Lie product \[\varphi \star \psi \coloneqq \overline{\C} \xrightarrow{\Delta_{(1)}} \overline{\C} \circ_{(1)} \overline{\C} \xrightarrow{\varphi \circ_{(1)} \psi} \End_A \circ_{(1)} \End_A \xrightarrow{\gamma_{(1)}} \End_A \ ,\] and the differential \[d (\varphi) = d_{\End_A} \circ \varphi - (- 1)^{|\varphi|} \varphi \circ d_{\overline{\C}} \ .\] \end{definition}

\noindent Within this framework, one relies on a colored version of the homotopy transfer theorem, proved in \cite{War19} for a characteristic zero field. In the case of set colored non-symmetric operads, one can deal with any coefficient ring see e.g. \cite[Theorem~5]{DV15}.

\section{\textcolor{bordeau}{ Minimal model for highly connected varieties.}}\label{section3}

In \cite{Zho22}, Zhou proved that for  a $k$-connected compact manifold $M$ of dimension $d$ smaller than $(\ell+1)k+2$ for $\ell \geqslant 3$, there exists a strictly unital $\mathcal{A}_{\infty}$-algebra structure of the form \begin{equation}\label{mod}
	(H_{dr}^*(M), \psi_2 , \dots, \psi_{\ell - 1}) \ .
\end{equation} which is related to $\Omega^*_{dr}(M)$ by an $\infty$-quasi-isomorphism. The purpose of this section is, first, to revisit the proof through the lens of obstruction classes, and second, to highlight that this proof extends beyond the characteristic zero setting and applies to other types of cohomology. In all this section, we work over a commutative ground ring $R$.

\begin{remark}
	The case $\ell = 3$ (corresponding to gauge formality)  was previously established by Miller in \cite{Mil79}. The case $\ell =5$ was proved by Crowley--Nordstr{$\ddot{o}$}m in \cite{CN20}. 	A more direct proof of Zhou's result was established by \cite{FH23}. They also prove a variant of this result, conjectured by Zhou, stating that if $d$ is smaller than $(\ell+1)k+ 4$ and the $(k+1)$-betti number $b_{k+1} = 1$ then 	there exists a strictly unital $\mathcal{A}_{\infty}$-algebra structure of the form (\ref{mod}) which is related to $\Omega^*_{dr}(M)$ by an $\infty$-quasi-isomorphism.
\end{remark}

\noindent In this section, we focus on $\mathcal{A}_{\infty}$-algebra structures, in which case the reduced conilpotent dg coproperad $\C$ of Section \ref{section2} is simply the Koszul dual cooperad $As^{\antishriek}$. The convolution Lie admissible algebra $\mathfrak{g}_A$ is the Hochschild cochain complex \[\mathfrak{g}_A = \prod_{k \geqslant 2} s^{-k + 1} \Hom \left(A^{\otimes k}, A \right)  \] with the usual Lie bracket and whose filtration is given by the arity, i.e. \[\mathcal{F}^n\mathfrak{g}_A = \prod_{k \geqslant n +1} s^{-k + 1} \Hom \left(A^{\otimes k}, A \right) \ . \]

\begin{definition}
	An $\mathcal{A}_{\infty}$-algebra $(A, \varphi)$ is \emph{strictly unital} if there exists  $1_A \in A^0$ such that
	\begin{itemize}
		\item[$\centerdot$] $d_A(1_A) = 0,$ $\varphi_2 (1_A, z_1) = \varphi_2(z_1, 1_A) = z_1,$
		\item[$\centerdot$] $\varphi_p (z_1, \dots, z_p)=0$, for all $p \geqslant 3$ when some $z_j = 1_A $. 
	\end{itemize}    
\end{definition}

\noindent The following theorem is a variant of \cite[Theorem 3.2]{Zho22}. While this result was originally stated for a field of characteristic zero, its proof actually extends to more general settings. We now present an alternative proof using obstruction theory that works for all rings. 

\begin{theorem}\label{5 hyp}
	Let $\ell \geqslant 3$ be such that $\ell$ and $\ell +1$ are units in the ground ring $R$. 
 Let $(A, \varphi)$ be a unital dg assocative algebra with $1_A$ not exact. Suppose that $A$ is related to $H(A)$ via a contraction. Let $n, k  \geqslant 1$ be such that 
	\begin{enumerate}
		\item  $n \leqslant (\ell +1)k +2$; 
		\item $A^i = 0$ if $i <0$ or $i > n$ and $H^i(A)= 0$ if $1 \leqslant i \leqslant k$ or $n-k \leqslant i \leqslant n - 1 $;
		\item $H^i(A)$ is a finitely generated free $R$-module with basis $\{x_1^i, \dots, x_{d_i}^i\}$ and $d_0 =  d_n = 1$; 
		\item there exists a dual basis $\{y_1^{n-i}, \dots, y_{d_i}^{n-i}\}$ of $H^{n-i}(A)$ such that \[ \varphi_*(y^{n-i}_u, x^i_v) = \delta_{uv} \nu \]
		where $\varphi_*$ is the induced structure on $H(A)$ and $\nu = x_1^n$ denotes a generator of $H^n(A)$;
	\end{enumerate} There exists a strictly unital $\mathcal{A}_{\infty}$-algebra structure of the form \[\left(H(A), \psi_2 , \dots, \psi_{\ell - 1} \right) \ , \] which is related to $(A, \varphi)$ by an $\infty$-quasi-isomorphism. 
\end{theorem}

\begin{remark}
	In characteristic zero, the proof of the theorem is an application of the obstruction theory set out above. However, it generalizes to any coefficient ring. In that case, it is not a direct application of obstruction sequences, as these do not work for arbitrary coefficient rings (except in the special case of formality). Nevertheless, the intuition provided by obstruction theory is used to establish the result.
\end{remark}

\begin{proof}
By the homotopy transfer theorem \ref{HTT}, the algebra $(A, \varphi)$ admits a transferred structure $(H(A), \varphi^1)$ which can be chosen strictly unital by \cite[Theorem~2.14]{Zho22} and such the following cyclic condition holds for all $p \geqslant 2$ and $(z_1, \dots, z_p)$ such that $|z_1| + \cdots + |z_{p}| = n + p - 2$,  \[ \sum_{j = 1}^{p} (-1)^{\alpha_j} \varphi^1_p (z_j, \dots , z_p , z_1 \dots, z_{j-1})  = 0 \ , \] where $\alpha_{j +1} \coloneqq j (\ell + 1 ) + (|z_1| + \cdots + |z_j|)(n - \ell + 1)$ by \cite[Lemma~4.2]{Zho19}. Let us first note that the assumptions imply that $\varphi^1_p = 0$, for all $p \geqslant \ell + 2$. Indeed, if a component $\varphi^1_p(z_1, \dots, z_p)$ is non-zero, it lies in degree at least $pk + 2$.  If $ p$ is greater than $\ell + 2 $, then $p k +2 > n$ and this implies that $\varphi^1_p = 0$. Let us set \[\psi \coloneqq  \varphi^1_2 + \cdots + \varphi^1_{\ell-1} \ ,\] which is a Maurer--Cartan element since the Maurer--Cartan equation of $\varphi^1$ is equivalent to $\psi \star \psi = 0$ by degree reasons. To conclude, we prove that $(H(A), \varphi^1)$ and $(H(A), \psi)$ are gauge homotopy equivalent. The first non-trivial obstruction is \[\vartheta_{\ell-1} = \left[\varphi^1_{\ell}\right] \in H_{-1} \left( \mathfrak{h}/ \mathcal{F}^{\ell}\mathfrak{h} \right)\] where $\mathfrak{h} \coloneqq \mathfrak{g}_{H(A)}^{\psi}$. Let us prove that $\vartheta_{\ell}$ vanishes, i.e. that there exists $\lambda \in  \mathfrak{h}_0$ such that \[ d^{\psi} \left(\lambda\right) \equiv \varphi^1_{\ell} \pmod{\mathcal{F}^{\ell }\mathfrak{h}} \ . \] First, we claim that \begin{equation}\label{z} \varphi^1_{\ell} (z_1, \dots, z_{\ell}) \neq 0 \quad  \implies \quad |z_1| + \cdots + |z_{\ell}| = n + \ell - 2   \ . \end{equation}
Indeed, if $\varphi^1_{\ell}(z_1, \dots, z_{\ell} )$ is non-zero, it is at least of degree $\ell k + 2$.  Since $\ell k +2 \geqslant n-k $, the element $\varphi^1_{\ell} \left(z_1, \dots, z_{\ell} \right)$ is non-zero only if it lies in $H^n(A)$, which proves Claim (\ref{z}).  For all $(z_1, \dots, z_{\ell} )$ such that $|z_1| + \cdots + |z_{\ell}| = n + \ell - 2$,  there exists $c(z_1, \dots, z_{\ell}) \in R$ such that \[\varphi^1_{\ell} \left(z_1, \dots, z_{\ell} \right) = c(z_1, \dots, z_{\ell}) \nu \ . \] Let us define an element $\lambda \in \mathfrak{h}_0$ of arity $\ell - 1$ as follows. For all $i_1 + \cdots + i_{\ell -1} < n + \ell -2$ such that $i_j > 0$ for all $j$,  let us set $s \coloneqq (n+ \ell -2) - (i_1 + \cdots + i_{\ell - 1})$ and \[\lambda \left(x^{i_1}_{r_1}, \dots,x^{i_{\ell -1}}_{r_{\ell -1}} \right) \coloneqq \sum_{t = 1}^{d_s} \sum_{j = 1}^{\ell - 1} (-1)^{\kappa(j)}  \frac{\ell - j}{\ell}\tilde{c}(j,t) y_t^{n-s} \ ,  \] where $\kappa(j) \coloneqq j(\ell + 1) + s(n -1) + (i_1 + \cdots + i_{j - 1})(n - \ell + 1) $ and \[\tilde{c}(j,t) \coloneqq c \left(x^{i_j}_{r_j },\dots,x^{i_{\ell -1}}_{r_{\ell -1}}, x^s_t, x^{i_1}_{r_1}, \dots,x^{i_{j -1}}_{r_{j -1}} \right) \ .\] In all other cases, we simply set $\lambda \coloneqq 0$. We have  \[ d^{\psi} \left(\lambda\right) \equiv \left[\varphi_*, \lambda\right] \pmod{\mathcal{F}^{\ell }\mathfrak{h}} \ , \] where $\varphi_* = \varphi_2^1$. Let us prove that $\left[\varphi_*, \lambda\right] =  \varphi^1_{\ell}  \ , $ where \[\left[\varphi_*, \lambda\right] = \varphi_* \star \lambda - \lambda\star \varphi_*  \ . \] By construction, this element is of arity $\ell$ and of degree $2 - \ell$ and it satisfies a property similar to $(\ref{z})$. For all $(z_1, \dots, z_{\ell} )$ such that $|z_1| + \cdots + |z_{\ell}| = n + \ell - 2$, we have \[\left(\lambda \star \varphi_*\right)(z_1, \dots, z_{\ell}) =0 \] by degree reasons.  It suffices to prove that for all $\left(x^{i_1}_{r_1}, \dots,x^{i_{\ell }}_{r_{\ell}} \right)  $, we have  \[ \left(\varphi_* \star \lambda \right)\left(x^{i_1}_{r_1}, \dots,x^{i_{\ell }}_{r_{\ell}} \right) = c\left(x^{i_1}_{r_1}, \dots,x^{i_{\ell }}_{r_{\ell}} \right) \nu \ . \]  The star product decomposes as $\varphi_* \star \lambda = \varphi_* \circ_2 \lambda + (-1)^{\ell - 2}  \varphi_* \circ_1 \lambda $.  On the one hand  \begin{equation*}
 	\begin{split}
 	 \left(\varphi_* \circ_1 \lambda \right) \left(x^{i_1}_{r_1}, \dots,x^{i_{\ell }}_{r_{\ell}} \right) & = \varphi_*\left(   \lambda \left(x^{i_1}_{r_1}, \dots,x^{i_{\ell -1}}_{r_{\ell -1}}\right), x^{i_{\ell}}_{r_{\ell}}   \right)  \\ & = \sum_{t = 1}^{d_s} \sum_{j = 1}^{\ell - 1} (-1)^{\kappa(j)}  \frac{\ell - j}{\ell}c(j,t)  \varphi_* \left( y_t^{n-s} ,x^{i_{\ell}}_{r_{\ell}}  \right)  \\ & =   \sum_{j = 1}^{\ell - 1} (-1)^{\kappa(j) + i_{\ell}\left(n -i_{\ell}\right) } \frac{\ell - j}{\ell}c(j,r_{\ell}) \nu  \end{split}
 \end{equation*}  where $\kappa(j) = j(\ell + 1) + i_{\ell}(n -1) + (i_1 + \cdots + i_{j - 1})(n - \ell + 1) \ . $ On the other hand, we have  \begin{equation*}
 	\begin{split}
 	\left(\varphi_* \circ_2 \lambda \right) \left(x^{i_1}_{r_1}, \dots,x^{i_{\ell }}_{r_{\ell}} \right) & =  (-1)^{i_1(2 - \ell)} \varphi_*\left(x^{i_1}_{r_1},   \lambda \left(x^{i_2}_{r_2}, \dots,x^{i_{\ell }}_{r_{\ell}}  \right) \right) \\ & = \sum_{j = 1}^{\ell - 1} (-1)^{\kappa'(j) + i_1 ( 2 - \ell) } \frac{\ell - j}{\ell} c'(j,r_{1}) \nu   
 	 \end{split}
 \end{equation*}  where $\kappa'(j) = j(\ell + 1) + i_{1}(n -1) + (i_2 + \cdots + i_{j })(n - \ell + 1) $ and \[c'(j, r_1) = c\left(x^{i_{j+1}}_{r_{j+1} },\dots , x^{i_{\ell}}_{r_{\ell}}, x^{i_1}_{r_1}, \dots,x^{i_{j}}_{r_{j }} \right) \ .\] Observe that when $ 1 \leqslant j \leqslant \ell -2 $, we have $c'(j, r_1) = \tilde{c}(j +1 , r_{\ell})$. We can also denote $\tilde{c}(\ell, r_{\ell}) \coloneqq c' (\ell - 1 , r_1 ) $. Let us set \[ \alpha_{j +1} \coloneqq j (\ell + 1 ) + (i_1 + \cdots + i_{j  })(n - \ell + 1) \ .  \] We have \[ (-1)^{\alpha_{j+1} } =  (-1)^{\kappa'(j  ) + i_1 ( 2 - \ell) } = -  (-1)^{\kappa(j+1) + \ell + i_{\ell}\left(n - i_{\ell}\right) } \ .  \] This leads to \[\left(\varphi_* \star \lambda \right)\left(x^{i_1}_{r_1}, \dots,x^{i_{\ell }}_{r_{\ell}} \right) =  \sum_{j = 1}^{\ell } (-1)^{\alpha_j } \frac{1}{\ell} \tilde{c}(j,r_{\ell}) \nu -  \tilde{c}(1,r_{\ell}) \nu \ ,   \] since $\alpha_1 = 0$.  Finally, we have \[\sum_{j = 1}^{\ell } (-1)^{\alpha_j } \frac{1}{\ell} \tilde{c}(j,r_{\ell}) \nu = \frac{1}{\ell} \sum_{j = 1}^{\ell } (-1)^{\alpha_j } \varphi_{\ell}^1 \left(x^{i_j}_{r_j },\dots,x^{i_{\ell -1}}_{r_{\ell -1}}, x^{i_{\ell}}_{r_{\ell}}, x^{i_1}_{r_1}, \dots,x^{i_{j -1}}_{r_{j -1}} \right) = 0 \ ,   \]   $\tilde{c}(1,r_{\ell}) \nu = c\left(x^{i_1}_{r_1}, \dots,x^{i_{\ell }}_{r_{\ell}} \right) \nu $ and the desired gauge is given by $- \lambda$. Let us set $f \coloneqq 1 - \lambda$ and  $\phi^2 \coloneqq f \cdot \varphi^1$. Since $\lambda$ is concentrated in arity $\ell -1$, we have \[\phi^2 \equiv \psi + \varphi^1_{\ell} - d^{\psi} (-\lambda) \equiv \psi \pmod{\mathcal{F}^{\ell} \mathfrak{h}} \ .  \]  We note that $\phi^2$ is still strictly unital by construction. The next obstruction is \[\vartheta_{\ell } = \left[ \phi_{\ell + 1}^2\right] \in H_{-1} \left( \mathfrak{h}/ \mathcal{F}^{\ell +1}\mathfrak{h} \right) \ . \] Let us prove that this second obstruction vanishes.  First, we claim  that $  \phi_{\ell + 1}^2 (z_1, \dots, z_{\ell +1})$ is none zero only if $ |z_i| = k + 1$ for all $i$. Thus, it is at least of degree $(\ell + 1)k + 2$. By the condition on $n$, this implies $n = (\ell + 1)k + 2$ and $ |z_i| = k + 1$ for all $i$. For all $(z_1, \cdots, z_{\ell})$ such that  $ |z_i| = k + 1$ for all $i$, there exists $b(z_1, \dots, z_{\ell +1 }) \in R$ such that \[\phi_{\ell + 1}^2 \left(z_1, \dots, z_{\ell +1} \right) = b(z_1, \dots, z_{\ell + 1}) \nu \ . \] Let us define an element $\beta \in \mathfrak{h}_0$ of arity $\ell$ as follows. Let us set $\beta \left(z_1, \dots, z_{\ell} \right) \coloneqq 0$ if there exists $i$ such that $|z_i| \neq k +1$ and \[\beta \left(x^{k+1}_{r_1}, \dots,x^{k+1}_{r_{\ell }} \right) \coloneqq \sum_{t = 1}^{d_{k+1}} \sum_{j = 1}^{\ell } (-1)^{\tilde{\kappa}(j)} \frac{\ell +1 - j}{\ell +1} \tilde{b}(j,t) y_t^{n-k-1} \ ,  \] where $\tilde{\kappa}(j) = jkl+(k+1)(l+1)$ and \[\tilde{b}(j, t) \coloneqq b \left(x^{k+1}_{r_{j}},\dots , x^{k+1}_{r_{\ell}}, x^{k+1}_{t} x^{k+1}_{r_1}, \dots,x^{k+1}_{r_{j -1 }} \right) \ .\] Let us prove that \[ d^{\psi} \left(\beta \right) \equiv  \phi_{\ell + 1}^2 \pmod{\mathcal{F}^{\ell + 2}\mathfrak{h}} \ . \] As before, it suffices to prove that for all $\left(x^{k+1}_{r_1}, \dots,x^{k+1}_{r_{\ell +1}} \right)  $, we have  \[ \left(\varphi_* \star \beta \right)\left(x^{k+1}_{r_1}, \dots,x^{k+1}_{r_{\ell + 1}} \right) = b \left(x^{k+1}_{r_1}, \dots,x^{k+1}_{r_{\ell +1}} \right) \nu \ . \]  The star product decomposes as $\varphi_* \star \beta = \varphi_* \circ_2 \beta + (-1)^{\ell - 1}  \varphi_* \circ_1 \beta $. On the one hand, we have  \begin{equation*}
 \begin{split}
 	\left(\varphi_* \circ_1 \beta \right) \left(x^{i_1}_{r_1}, \dots,x^{i_{\ell +1 }}_{r_{\ell +1}} \right) & =   \sum_{j = 1}^{\ell } (-1)^{\tilde{\kappa}(j) +  (k +1) (n - k - 1) } \frac{\ell + 1 - j}{\ell +1}\tilde{b}(j,r_{\ell +1}) \nu \ . \end{split}
 	\end{equation*} On the other hand, we have  \begin{equation*}
	\begin{split}
		\left(\varphi_* \circ_2 \beta \right) \left(x^{k+1}_{r_1}, \dots,x^{k+1}_{r_{\ell +1}} \right) & =   \sum_{j = 1}^{\ell } (-1)^{ \tilde{\kappa}(j) + (1 - \ell)(k+1) } \frac{\ell +1 - j}{\ell - 1} \tilde{b}(j + 1,r_{\ell +1}) \nu   \ . 
	\end{split}
\end{equation*} Let us set \[\tilde{\alpha}_{j+1} = (j-1)(\ell +2) + (j -1)(k+1)(n- \ell) \ . \] Since $n = (\ell+  1) k +2 $, we have \[(-1)^{\tilde{\alpha}_{j+1}} = (-1)^{ \tilde{\kappa}(j) + (1 - \ell)(k+1) }  = -  (-1)^{\ell-1 + \tilde{\kappa}(j +1) +  (k +1) (n - k - 1) }  \ .  \] As before, since $\alpha_1 = 0$, this leads to \[\left(\varphi_* \star \beta \right)\left(x^{k+1}_{r_1}, \dots,x^{k+1}_{r_{\ell +1}} \right) =  \sum_{j = 1}^{\ell+1 } (-1)^{\tilde{\alpha}_j } \frac{1}{\ell +1} \tilde{b}(j,r_{\ell +1}) \nu -  \tilde{b}(1,r_{\ell +1}) \nu \ .   \]   Finally, we claim that \cite[Lemma~4.2]{Zho19} applies once again to $(H(A), \phi^2)$ so that \[\sum_{j = 1}^{\ell+1 } (-1)^{\tilde{\alpha}_j }  \tilde{b}(j,r_{\ell +1}) \nu =  \sum_{j = 1}^{\ell+1 } (-1)^{\tilde{\alpha}_j }  \phi_{\ell + 1}^2\left(x^{k+1}_{r_{j}},\dots , x^{k+1}_{r_{\ell}}, x^{k+1}_{r_{\ell +1}} x^{k+1}_{r_1}, \dots,x^{k+1}_{r_{j -1 }} \right) = 0 \ . \] One can actually apply \cite[Lemma~4.2]{Zho19} to the model $M(A) = (H(A), \phi^2)$ together with the quasi-isomorphism $\tilde{f} \coloneqq (1- \lambda) \circledcirc p_{\infty}$ since the construction \[\phi^2 \coloneqq f \cdot \varphi^1 \quad \mbox{and} \quad \tilde{f} = (1- \lambda) \circledcirc p_{\infty}  \]  makes it appears precisely as a structure obtained through the induction process of \cite[Theorem~2.14]{Zho22}. Let us set $g \coloneqq 1 - \beta$ and  $\phi^3 \coloneqq g \cdot \phi^2$.  We have \[\phi^3 \equiv \psi + \phi^2_{\ell+1} - d^{\psi} (-\beta) \equiv \psi \pmod{\mathcal{F}^{\ell+1} \mathfrak{h}} \ .  \]  Furthermore, a component $\phi^3_p(z_1, \dots, z_p)$ is possibly none zero for $ p \geqslant \ell + 2 $ then it is in degree at least $p k +2$. Since $p k +2 > n$, this implies that  \[(g \circledcirc f) \cdot \varphi^1 = \psi \ . \]  This concludes the proof.   
\end{proof}

\begin{theorem}\label{poincaré}
Let $\ell \geqslant 3$ and let $\mathbb{K}$ be a field such that $\ell$ and $\ell +1 $ are units in $\mathbb{K}$. Let $M$ be a compact $k$-connected $C^{\infty}$-manifold whose dimension $n$ is smaller than $(\ell+1)k+2$. Its singular cochains $C^{*}_{\mathrm{sing}}(M, \mathbb{K})$ has an $A_{\infty}$-minimal model whose arity $p$ component vanishes for all $p \geqslant \ell$. 
\end{theorem}

\begin{proof}
	If $M$ is $\mathbb{K}$-oriented, the result follows from Theorem \ref{5 hyp}. The hypotheses $(4)$ is satisfied by Poincaré duality. When $M$ is not $\mathbb{K}$-orientable, then $H^n(M; \mathbb{K}) = 0$. Let $\tilde{M}$ be the orientation bundle over $M$. It is also $k$-connected and by twisted Poincaré duality \[H^{n- i} (M; \mathbb{K}) \cong H_i(\tilde{M}; \mathbb{K}) = 0 \ .  \] In particular, we have  $H^i(\tilde{M}; \mathbb{K})= 0$ if $1 \leqslant i \leqslant k$ or $ i \geqslant n-k $. By the homotopy transfer theorem \ref{HTT}, the algebra $(A, \varphi)$ admits a transferred structure $(H(A), \varphi^1)$, which can be chosen strictly unital by \cite[Theorem~2.14]{Zho22}. Furthermore, the assumptions imply that $\varphi^1_p = 0$, for all $p \geqslant \ell$. Indeed, if a component $\varphi^1_p(z_1, \dots, z_p)$ is non-zero, it lies in degree at least $pk + 2$.  If $ p$ is greater than $\ell + 2 $, then $p k +2 \geqslant n - k$ and this implies that $\varphi^1_p = 0$.
\end{proof}

\begin{theorem} \label{poincaré2}
Let $q$ be a prime number. Let $K$ be a separately closed field in which $q$ is invertible. Let $\ell \geqslant 3$ be such that $\ell$ and $\ell +1 $ are units in $\mathbb{F}_q$. Let $X$ be a $k$-connected irreducible, proper, smooth variety defined over $K$ whose dimension is smaller than $(\ell+1)k+2$.  Its étale cochains $C_{\mathrm{\acute{e}t}}^*(X, \mathbb{F}_{q})$ has an $A_{\infty}$-minimal model whose arity $p$ component vanishes for all $p \geqslant \ell$. 
\end{theorem}

\begin{proof}
	Any smooth variety defined over a separately closed field is orientable.  The result follows from Theorem \ref{5 hyp}. The hypotheses $(4)$ is satisfied by Poincaré duality \cite[Proposition 2.5]{EW16}.
\end{proof}

\bibliographystyle{alpha}
\bibliography{bib}

\end{document}